\documentclass[11pt]{scrartcl}

\usepackage{url}
\usepackage[hidelinks]{hyperref}
\usepackage[utf8]{inputenc}
\usepackage[small]{caption}
\usepackage{graphicx}
\usepackage{amsmath}
\usepackage{booktabs}
\usepackage{amsthm}
\usepackage{amssymb}

\usepackage{cleveref}
\usepackage{color}
\usepackage{xspace}
\usepackage{tikz}
\usepackage{tabularx}
\usepackage{pbox}
\usepackage{framed}
\usepackage[shortlabels]{enumitem}

\usepackage{subcaption}
\usepackage{pgfplots}

\usetikzlibrary{arrows.meta}
\tikzset{>={Latex[width=1.5mm,length=1.5mm]}}

\usepackage{multirow}

\DeclareMathOperator*{\E}{\mathbb{E}}

\newtheorem{theorem}{Theorem}[section]

\newtheorem{lemma}[theorem]{Lemma}

\theoremstyle{definition}
\newtheorem{definition}[theorem]{Definition}

\allowdisplaybreaks

\title{Generalized Kings and Single-Elimination Winners in Random Tournaments}

\author{
Pasin Manurangsi\\Google Research
\and
Warut Suksompong\\National University of Singapore
}

\date{\vspace{-5ex}}

\begin{document}

\maketitle

\begin{abstract}
Tournaments can be used to model a variety of practical scenarios including sports competitions and elections.
A natural notion of strength of alternatives in a tournament is a generalized king: an alternative is said to be a \emph{$k$-king} if it can reach every other alternative in the tournament via a directed path of length at most $k$.
In this paper, we provide an almost complete characterization of the probability threshold such that all, a large number, or a small number of alternatives are $k$-kings with high probability in two random models.
We show that, perhaps surprisingly, all changes in the threshold occur in the range of constant $k$, with the biggest change being between $k=2$ and $k=3$.
In addition, we establish an asymptotically tight bound on the probability threshold for which all alternatives are likely able to win a single-elimination tournament under some bracket.
\end{abstract}  

\section{Introduction}

Social choice theory is the study of how to aggregate individual preferences and opinions of agents on a set of alternatives in order to reach a collective decision.
In many practical situations, the relationship between the alternatives is represented by a \emph{dominance relation}, which specifies the relative strength of the alternatives in any pairwise comparison.
For example, in sports competitions the dominance relation signifies the match outcome when two players or teams play each other, while in elections the relation represents the pairwise majority comparisons among the candidates.
The structure consisting of the alternatives and their dominance relation is called a \emph{tournament}, and the analysis of tournament winner selection methods---also known as \emph{tournament solutions}---has received significant attention from researchers in the past few decades \cite{BrandtBrHa16,Laslier97}.

Among the vast array of tournament solutions proposed in the literature, two of the earliest and best-known ones are the \emph{top cycle} \cite{Good71,Miller77,Schwartz72} and the \emph{uncovered set} \cite{Fishburn77,Miller80}.
An alternative belongs to the top cycle if it can reach every other alternative via a directed path in the tournament.
Note that if the tournament contains $n$ alternatives, any such path has length $n-1$ or less (the length of a path refers to the number of edges in the path).\footnote{The bound $n-1$ cannot be improved. To see this, consider a tournament with alternatives $x_1,\dots,x_n$ such that $x_i$ dominates $x_j$ if $i-j \ge 2$ or $j-i = 1$. Alternative $x_1$ can reach every other alternative, but it cannot reach $x_n$ via a path of length $n-2$ or less.}
Similarly, the uncovered set---also known as the set of \emph{kings} \cite{Maurer80}---consists of the alternatives that can reach every other alternative via a path of length at most two.
It is clear from the definitions that the uncovered set is always a subset of the top cycle.
Moreover, both tournament solutions can be viewed as special cases of a generalized notion of kings called \emph{$k$-kings}, which correspond to the alternatives that can reach every other alternative via a path of length at most $k$.
Indeed, the uncovered set is the set of $2$-kings, while the top cycle contains precisely the $(n-1)$-kings.
Examples of tournaments and $k$-kings can be found in Section~\ref{sec:prelim}.

Given that tournament solutions are meant to distinguish the best alternatives from the rest, it is natural to ask how selective each tournament solution is.
Moon and Moser~\cite{MoonMo62} and Fey~\cite{Fey08} addressed this question and showed that all alternatives are likely to be $2$-kings, and hence also $(n-1)$-kings, when the tournament is large.
In particular, their results hold under the \emph{uniform random model}, wherein each edge is oriented in one direction or the other with equal probability independently of other edges.
Saile and Suksompong~\cite{SaileSu20} extended these results to the \emph{generalized random model}, in which the orientation of each edge is determined by probabilities within the range $[p, 1-p]$ for some parameter $p \le 1/2$, and these probabilities may vary across edges.
The generalized random model allowed these authors to demonstrate a difference between the two tournament solutions---while all alternatives are likely to be $(n-1)$-kings as long as $p\in\omega(1/n)$, the same is true for $2$-kings only when $p\in\Omega(\sqrt{\log n/n})$, so the two thresholds differ by roughly $\Theta(\sqrt{n})$.
This raises the following question: How does the probability threshold change as we transition from $k=2$ to $k=n-1$?
Does it already decrease at around $k = \sqrt{n}$, or does it remain the same until, say, $k \approx n/2$?

\renewcommand{\arraystretch}{1.3}
\begin{table*}[!t]
\centering
    \begin{tabular}{| c | c | c | c |}
    \hline
    \multicolumn{2}{ |c| }{Tournament solution} & Condorcet random model  & Generalized random model   \\ \hline \hline
    \multirow{4}{*}{$k$-kings} & $k=2$ & $\Omega(\sqrt{\log n/n})$ \cite{SaileSu20} & $\Omega(\sqrt{\log n/n})$ \cite{SaileSu20} \\ \cline{2-4}
    & $3\leq k\leq 4$ & $\Omega(\log n/n)$ (Thm.~\ref{thm:3king-positive}, \ref{thm:4king-negative}) & $\Omega(\log n/n)$ (Thm.~\ref{thm:3king-positive}, \ref{thm:4king-negative}) \\ \cline{2-4}
    & $k=5$ & $\omega(1/n)$ (Thm.~\ref{thm:5king-positive-Condorcet}) & $\Omega(\log\log n/n)$ (Thm.~\ref{thm:5king-positive-generalized}) \\ \cline{2-4}
    & $6\leq k\leq n-1$ & $\omega(1/n)$ (Thm.~\ref{thm:6king-positive}) & $\omega(1/n)$ (Thm.~\ref{thm:6king-positive}) \\ \hline
    \multicolumn{2}{ |c| }{Single-elimination winners} & $\Omega(\log n/n)$ \cite{KimSuVa17} & $\Omega(\log n/n)$ (Thm.~\ref{thm:single-elimination}) \\ \hline
    \end{tabular}
    \vspace{5mm}
    \caption{Summary of the bounds on the probability $p$ at which the respective tournament solutions select all alternatives with high probability under the corresponding random model. All bounds are asymptotically tight except the bound for $k=5$ with the generalized random model, where there is a gap between $\Omega(\log\log n/n)$ and $\omega(1/n)$. The results with $\Omega(\cdot)$ hold when the associated constant term is sufficiently large.}
    \label{table:summary}
\end{table*}

In this paper, we show that, perhaps surprisingly, all of the changes in the probability threshold occur when $k$ is constant.
In fact, when $k=6$, all alternatives are already likely to be $k$-kings provided that $p\in\omega(1/n)$, the same threshold as $k = n-1$---this significantly strengthens the result of Saile and Suksompong~\cite{SaileSu20} on $(n-1)$-kings.
For $k=3$ and $4$, we establish an asymptotically tight bound of $p\in \Omega(\log n/n)$, while for $k=5$ we leave the only (small) gap between $\Omega(\log\log n/n)$ and $\omega(1/n)$.
Besides the generalized random model, we consider a more specific model which has nevertheless been studied in several papers called the \emph{Condorcet random model} \cite{Frank68,KimSuVa17,LuczakRuGr96,Vassilevskawilliams10}.
In this model, there is a parameter $p$ and a linear order of alternatives from strongest to weakest, and the probability that a stronger alternative dominates a weaker one is $1-p$, independently of other pairs of alternatives.\footnote{The uniform random model corresponds to taking $p=1/2$. For any $p$, the Condorcet random model with parameter $p$ is a special case of the generalized random model with the same $p$---see Figure~\ref{fig:models} for an illustration.
Hence, a positive result for the generalized random model carries over to the Condorcet random model, while a negative result transfers in the opposite direction.}
For the Condorcet random model, we show that the threshold for $k=5$ is $\omega(1/n)$, whereas the thresholds for other values of $k$ remain tight.
Our results are summarized in Table~\ref{table:summary} and presented in \textbf{Section~\ref{sec:generalized-kings}}.
Taken together, they reveal the intriguing facts that (i) $2$-kings are distinctly more selective than $k$-kings for $k\geq 3$; (ii) $3$-kings and $4$-kings are slightly more selective than higher-order kings; and (iii) there is virtually no difference in discriminative power from $k=5$ all the way to $k=n-1$.

In addition to $k$-kings, we also consider the set of \emph{single-elimination winners}, which are alternatives that can win a (balanced) single-elimination tournament under some bracket, where the match outcomes in the single-elimination tournament are determined according to the dominance relation in the original tournament.\footnote{Following prior work on single-elimination tournaments, we assume that the number of alternatives is a power of two.}
Kim et al.~\cite{KimSuVa17} showed that all alternatives are likely to be single-elimination winners in the Condorcet random model as long as $p\in \Omega(\log n/n)$, and this bound is tight.\footnote{If $p\in o(\log n/n)$, the weakest alternative dominates $o(\log n)$ alternatives in expectation; this is insufficient since winning $\log_2 n$ matches is required to win a single-elimination tournament.}
For the generalized random model, they established an analogous statement in the range $p\in \Omega(\sqrt{\log n/n})$.
We close this gap by proving that even for the generalized random model, $p\in\Omega(\log n/n)$ already suffices for all alternatives to be single-elimination winners with high probability; moreover, a winning bracket for each alternative can be computed in polynomial time.
Our result, which can be found in \textbf{Section~\ref{sec:SE-winners}}, further lends credence to the observation that real-world tournaments can be easily manipulated \cite{MatteiWa16}.

In \textbf{Section~\ref{sec:number-kings}}, we move beyond the question of when \emph{all} alternatives are likely to be selected by a tournament solution, and instead ask when this is the case for a large or small number of alternatives.
Even though $p\in\Omega(\sqrt{\log n/n})$ is required in order for all alternatives to be $2$-kings with high probability \cite{SaileSu20}, we show that most of them are already likely to be $2$-kings as long as $p\in\Omega(\log n/n)$.
This threshold is exactly where the transition occurs: if $p\in O(\log n/n)$, we prove that it is almost surely the case that only a small fraction of the alternatives are $2$-kings.
Furthermore, for any $k\ge 3$, we establish that a large fraction of alternatives are likely to be $k$-kings provided that $p\in \omega(1/n)$.
These results illustrate the probability range under which each tournament solution is discriminative, and again exhibit a clear difference between $2$-kings and higher-order kings.
Finally, in \textbf{Section~\ref{sec:experiments}}, we complement our theoretical findings with results from computer experiments.

Before proceeding further, let us provide some discussion and intuition on our models and results.
The uniform random model, while being the most basic choice to start probabilistic investigations of tournaments with, is clearly unrealistic since alternatives in most applications have different levels of strength.
The Condorcet random model addresses this shortcoming by allowing some alternatives to be stronger than others.
However, as Saile and Suksompong~\cite{SaileSu20} pointed out, the Condorcet random model still suffers from limitations such as using the same probability for all pairs of alternatives (regardless of the \emph{extent} to which one alternative is stronger than the other) or not allowing for ``bogey teams'' (i.e., weak teams that often beat certain stronger teams).
The generalized random model only assumes that each match is sufficiently random, and therefore does not have these limitations.
The fact that there are usually a large number of $k$-kings when the tournaments are drawn from the uniform random model, as well as from the Condorcet and generalized random models for sufficiently large $p$, can be explained intuitively by noting that, for large tournaments, there are a high number of potential paths through with one alternative can reach another.
As we will see later in our experimental results (Section~\ref{sec:experiments}), when the tournament size is roughly $100$, the fraction of $k$-kings is already very high for most parameters and models.

\subsection{Related Work}

Tournament solutions have been extensively studied for the past several decades from the axiomatic \cite{BrandtBrSe18,BrandtFi07,HanVa19,LaffondLaLe93,LaffondLaLe94}, computational \cite{AzizBrFi15,BrandtBrSe11,BrillScSu22,Dey17,MnichShYa15}, and probabilistic perspectives \cite{Fey08,ScottFe12}.\footnote{We refer to the surveys by Laslier~\cite{Laslier97} and Brandt et al.~\cite{BrandtBrHa16} for extensive overviews as well as one by Suksompong~\cite{Suksompong21} for recent developments.}
This strong interest is due to the fact that tournaments can be used to model a variety of scenarios including voting, webpage ranking, and biological interactions---indeed, the well-known PageRank algorithm can be viewed as a form of tournament solution \cite{BrandtFi07}, and the pecking order in a flock of chickens can be represented by a tournament \cite{Maurer80}.
In an early work on tournament solutions, Maurer \cite{Maurer80} characterized the pairs $(n,r)$ for which there exists a tournament of size $n$ such that the number of $2$-kings is exactly $r$, and proved that all alternatives are $2$-kings with high probability in the uniform random model.\footnote{Maurer wrote: ``That it is very rare for only one chicken to be king did not surprise me, but I was shocked to learn that the opposite extreme---every chicken is a king---is very common. I would never have guessed such a thing, let alone proved it, had not a student reported to my class the results of a computer program he ran which counted the number of kings in $100$ random flocks of $16$ chickens. No flock had fewer than $8$ kings, and most had $14$, $15$, or $16$!''}
There are several containment relations among common tournament solutions.
For example, the Copeland set, Slater set, Markov set, and Banks set are all contained in the set of $2$-kings (i.e., the uncovered set), which is in turn contained in the set of $(n-1)$-kings (i.e., the top cycle)---this provides a range of options in terms of discriminative power and other properties.
Even though $k$-kings admit a simple and elegant definition generalizing both $2$-kings and $(n-1)$-kings, and have attracted interest from graph theorists \cite{BrcanovPe10,PetrovicTh91,Tan06}, as far as we know, they have not been studied in the social choice context until recently.
Kim and Vassilevska Williams~\cite{KimVa15} and Kim et al.~\cite{KimSuVa17} identified conditions under which a $3$-king can win a single-elimination tournament.
Brill et al.~\cite{BrillScSu22} showed that computing the ``margin of victory'' of $k$-kings can be done efficiently for $k\leq 3$ but becomes NP-hard for $k\geq 4$.
They also illustrated through experiments that the margin of victory of $3$-kings behaves much more similarly to that of $(n-1)$-kings than to the corresponding notion of $2$-kings; our results therefore complement theirs by exhibiting that analogous behavior can be observed with respect to discriminative power.

As we mentioned earlier, the study of tournament solutions under the probabilistic lens was initiated by Moon and Moser~\cite{MoonMo62}, who showed that all alternatives are likely to be $(n-1)$-kings when the tournament is generated by the uniform random model.
Fey~\cite{Fey08} and Scott and Fey~\cite{ScottFe12} established the same property for two other tournament solutions, the Banks set and the minimal covering set, while Fisher and Ryan~\cite{FisherRy95} showed that the bipartisan set includes half of the alternatives on average.\footnote{The \emph{Banks set} is the set of alternatives that appear as the maximal element of some transitive subtournament that cannot be extended.
The \emph{bipartisan set} is the set of alternatives chosen with positive probability in the (unique) Nash equilibrium of the symmetric zero-sum game induced by the tournament.
Given a tournament and two alternatives $x$ and $y$, $x$ is said to \emph{cover} $y$ if (i) $x$ dominates $y$, and (ii) $x$ dominates every alternative $z$ that $y$ dominates.
The \emph{minimal covering set} is the (unique) smallest set of alternatives $B$ such that every alternative $x\not\in B$ is covered by some alternative in $B$ in the tournament restricted to $B\cup\{x\}$.
For further details about these tournament solutions, please see the surveys by Laslier~\cite{Laslier97} and Brandt et al.~\cite{BrandtBrHa16}.
} 
Brandt et al.~\cite{BrandtBrSe18} showed that any tournament solution that satisfies an attractive property called \emph{stability}, including the set of $(n-1)$-kings, the minimal covering set, and the
bipartisan set, must choose at least half of the alternatives on average.
The Condorcet random model has been analyzed, among others, by Frank~\cite{Frank68}, \L{}uczak et al.~\cite{LuczakRuGr96}, Vassilevska Williams~\cite{Vassilevskawilliams10}, and Kim et al.~\cite{KimSuVa17}, with the last paper also proposing the generalized random model.

Finally, single-elimination tournaments have constituted a popular topic of study in the past decade; see the survey by Vassilevska Williams~\cite{Vassilevskawilliams16} and a more recent one by Suksompong~\cite{Suksompong21}.
In particular, even though the problem of determining whether an alternative can win a single-elimination tournament is known to be NP-hard \cite{AzizGaMa18}, a wide range of algorithmic and complexity results have been developed by this active line of work \cite{ChatterjeeIbTk16,GuptaRoSa18,GuptaRoSa18-2,GuptaRoSa19,KimSuVa17,KimVa15,KonickiVa19,ManurangsiSu22,RamanujanSz17,StantonVa11,StantonVa11-2,Vassilevskawilliams10,VuAlSh09}.
For instance, Kim and Vassilevska Williams~\cite{KimVa15} gave an algorithm running in time $O(2^n\text{poly}(n))$ that not only decides whether a given alternative can win a single-elimination tournament but also counts the number of brackets for which this alternative is the winner.
Kim et al.~\cite{KimSuVa17} presented several conditions under which an alternative is guaranteed to be a single-elimination winner and its winning bracket is computable in polynomial time.
Ramanujan and Szeider~\cite{RamanujanSz17} and Gupta et al.~\cite{GuptaRoSa18,GuptaRoSa19} studied this problem from the perspective of parameterized complexity.
Issues of bribery and scheduling manipulation have also been extensively studied with respect to single-elimination as well as other tournament formats such as round-robin \cite{BaumeisterHo21,GuptaRoSa18-2,GuptaRoSa19,KimVa15,KonickiVa19,MatteiGoKl15,Suksompong16}.

\section{Preliminaries}
\label{sec:prelim}

A tournament $T$ consists of a set $V=\{x_1,\dots,x_n\}$ of vertices, also called \emph{alternatives}, and a set $E$ of directed edges.
For any two alternatives $x_i,x_j\in V$, there exists either an edge from $x_i$ to $x_j$ or an edge from $x_j$ to $x_i$, but not both.
The edges represent a \emph{dominance relation} between the alternatives: an edge from $x_i$ to $x_j$ means that $x_i$ \emph{dominates} $x_j$, a relation which we denote by $x_i\succ x_j$.
The \emph{outdegree} (resp., \emph{indegree}) of an alternative $x_i$ is the number of alternatives that $x_i$ dominates (resp., that dominate $x_i$).
We extend the dominance relation to sets of alternatives: for $V_1,V_2\subseteq V$, we write $V_1\succ V_2$ to mean that $x\succ x'$ for all $x\in V_1$ and $x'\in V_2$, and $V_1\succ x'$ to mean that $x\succ x'$ for all $x\in V_1$.
A set $V'\subseteq V$ is called a \emph{dominating set} if for every $x\in V\setminus V'$, there exists $x'\in V'$ such that $x'\succ x$.

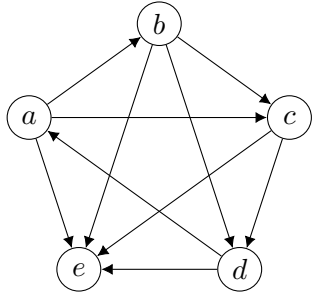
\begin{figure}[!ht]
\centering
\begin{tikzpicture}[scale=1.5]
\tikzstyle{every node}=[fill=white,circle,draw,minimum size=1.5em,inner sep=0pt]
		]
  \node (1) at (162:1.2) {$a$};
  \node (2) at (90: 1.2) {$b$};
  \node (3) at (18: 1.2) {$c$};
  \node (4) at (306:1.2) {$d$};
  \node (5) at (234:1.2) {$e$};
  \draw[->] (1) to (2);
  \draw[->] (1) to (3);
  \draw[->] (4) to (1);
  \draw[->] (1) to (5);
  \draw[->] (2) to (3);
  \draw[->] (2) to (4);
  \draw[->] (2) to (5);
  \draw[->] (3) to (4);
  \draw[->] (3) to (5);
  \draw[->] (4) to (5);
\end{tikzpicture}
\caption{A tournament with five alternatives. 
The outdegrees of $a,b,c,d,e$ are $3,3,2,2,0$, respectively. 
The alternatives $a$, $b$, and $d$ are $2$-kings, $c$ is a $3$-king but not a $2$-king, while $e$ is not a $k$-king for any $k$.
} 
\label{fig:tsol-example}
\end{figure}

We can now define the key notions of this paper.
\begin{itemize}
\item For any integer $k\geq 2$, an alternative is said to be a \emph{$k$-king} if it can reach every other alternative via a directed path of length at most $k$.
An example is shown in Figure~\ref{fig:tsol-example}.
\item Suppose that $n=2^r$ for some nonnegative integer $r$.
An alternative is said to be a \emph{single-elimination winner} if it wins a (balanced) single-elimination tournament under some bracket, where the outcome of each match is determined according to the dominance relation.
Formally, a single-elimination tournament is represented by a balanced binary tree with $n$ leaves corresponding to the alternatives; the assignment of the alternatives to the leaves is called a \emph{bracket}.
The winner of the tournament is determined recursively: the
winner of a leaf is the alternative at the leaf, and the winner of a subtree rooted at node $u$ is the winner of the match between
the winners of the subtrees rooted at the two children of $u$.
An example is shown in Figure~\ref{fig:SE-example}.
\end{itemize}

\begin{figure}[!ht]
\centering
\begin{tikzpicture}[scale=1]
\draw (2,2) -- (2.5,3) -- (3,2);
\draw (4,2) -- (4.5,3) -- (5,2);
\draw (6,2) -- (6.5,3) -- (7,2);
\draw (8,2) -- (8.5,3) -- (9,2);
\draw (2.5,3) -- (3.5,4) -- (4.5,3);
\draw (6.5,3) -- (7.5,4) -- (8.5,3);
\draw (3.5,4) -- (5.5,5) -- (7.5,4);
\draw[fill=white] (2,2) circle  [radius=0.25];
\draw[fill=white] (3,2) circle  [radius=0.25];
\draw[fill=white] (4,2) circle  [radius=0.25];
\draw[fill=white] (5,2) circle  [radius=0.25];
\draw[fill=white] (6,2) circle  [radius=0.25];
\draw[fill=white] (7,2) circle  [radius=0.25];
\draw[fill=white] (8,2) circle  [radius=0.25];
\draw[fill=white] (9,2) circle  [radius=0.25];
\node at (2,2) {$a$};
\node at (3,2) {$b$};
\node at (4,2) {$c$};
\node at (5,2) {$d$};
\node at (6,2) {$e$};
\node at (7,2) {$f$};
\node at (8,2) {$g$};
\node at (9,2) {$h$};
\draw[fill=white] (2.5,3) circle  [radius=0.25];
\draw[fill=white] (4.5,3) circle  [radius=0.25];
\draw[fill=white] (6.5,3) circle  [radius=0.25];
\draw[fill=white] (8.5,3) circle  [radius=0.25];
\node at (2.5,3) {$b$};
\node at (4.5,3) {$c$};
\node at (6.5,3) {$f$};
\node at (8.5,3) {$h$};
\draw[fill=white] (3.5,4) circle  [radius=0.25];
\draw[fill=white] (7.5,4) circle  [radius=0.25];
\node at (3.5,4) {$c$};
\node at (7.5,4) {$h$};
\draw[fill=white] (5.5,5) circle  [radius=0.25];
\node at (5.5,5) {$c$};
\end{tikzpicture}
\caption{A single-elimination tournament with eight alternatives, where we assume that $b\succ a$, $c\succ d$, $f\succ e$, $h\succ g$, $c\succ b$, $h\succ f$, and $c\succ h$; all other dominance relations are arbitrary.  Under this bracket, alternative $c$ is the single-elimination winner.}
\label{fig:SE-example}
\end{figure}
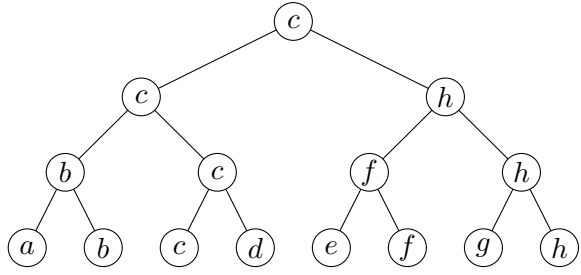

We will consider two random models for generating tournaments.
In the \emph{Condorcet random model}, there is a parameter $0\leq p\leq 1/2$. 
For $i<j$, alternative $x_j$ dominates $x_i$ with probability $p$ (so $x_i$ dominates $x_j$ with probability $1-p$), independently of other pairs of alternatives.
In the \emph{generalized random model}, there is a parameter $p_{i,j}$ for each pair $i\neq j$, where $p_{i,j} + p_{j,i} = 1$.
For any pair $i,j$, alternative $x_i$ dominates $x_j$ with probability $p_{i,j}$, independently of other pairs.
We will generally allow each probability $p_{i,j}$ to be chosen from the range $[p,1-p]$ for a given parameter $0\le p\le 1/2$.
See Figure~\ref{fig:models} for an illustration.
Before a tournament is generated from the generalized random model, the \emph{expected outdegree} of $x_i$ is defined as $\sum_{j\ne i}p_{i,j}$.
Following standard terminology in probability theory, we say that an event whose probability depends on $n$ occurs ``with high probability'' if the probability that it occurs approaches $1$ as $n\rightarrow\infty$.

\begin{figure}[!ht]
\centering
\begin{tikzpicture}
\draw[fill=white] (1,4) circle  [radius=0.25]; 
\draw[fill=white] (4,4) circle  [radius=0.25]; 
\draw[fill=white] (1,1) circle  [radius=0.25]; 
\draw[fill=white] (4,1) circle  [radius=0.25]; 
\node at (1,4) {$x_1$};
\node at (4,4) {$x_2$};
\node at (1,1) {$x_3$};
\node at (4,1) {$x_4$};
\draw[->] (1.25,4) -- (3.75,4); 
\draw[->] (1,3.75) -- (1,1.25); 
\draw[->] (1.18,3.82) -- (3.82,1.18); 
\draw[->] (3.82,3.82) -- (1.18,1.18); 
\draw[->] (4,3.75) -- (4,1.25); 
\draw[->] (1.25,1) -- (3.75,1); 
\node at (2.5,4.3) {$0.8$};
\node at (0.65,2.5) {$0.8$};
\node at (2.1,3.4) {$0.8$};
\node at (2.1,1.6) {$0.8$};
\node at (4.35,2.5) {$0.8$};
\node at (2.5,0.7) {$0.8$};

\draw[fill=white] (7,4) circle  [radius=0.25]; 
\draw[fill=white] (10,4) circle  [radius=0.25]; 
\draw[fill=white] (7,1) circle  [radius=0.25]; 
\draw[fill=white] (10,1) circle  [radius=0.25]; 
\node at (7,4) {$x_1$};
\node at (10,4) {$x_2$};
\node at (7,1) {$x_3$};
\node at (10,1) {$x_4$};
\draw[->] (7.25,4) -- (9.75,4); 
\draw[->] (7,3.75) -- (7,1.25); 
\draw[->] (7.18,3.82) -- (9.82,1.18); 
\draw[->] (9.82,3.82) -- (7.18,1.18); 
\draw[->] (10,3.75) -- (10,1.25); 
\draw[->] (7.25,1) -- (9.75,1); 
\node at (8.5,4.3) {$0.6$};
\node at (6.65,2.5) {$0.7$};
\node at (8.1,3.4) {$0.8$};
\node at (8.1,1.6) {$0.2$};
\node at (10.35,2.5) {$0.5$};
\node at (8.5,0.7) {$0.35$};
\end{tikzpicture}
\caption{Examples of probabilities for generating tournaments according to the Condorcet random model (left) and the generalized random model (right) with $p = 0.2$.
For each pair $i < j$, the corresponding number in the figure indicates the probability $p_{i,j}$ with which alternative $x_i$ dominates $x_j$; the reverse domination occurs with probability $1 - p_{i,j}$.
} 
\label{fig:models}
\end{figure}
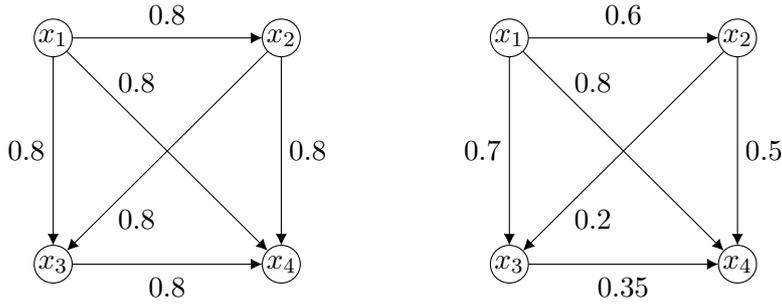

We now list two well-known probabilistic statements that will be used multiple times in this paper.
The first statement provides an upper bound on the probability that a sum of independent random variables is far from its expectation.

\begin{lemma}[Chernoff bound] \label{lem:chernoff}
Let $X_1, \dots, X_k$ be independent random variables taking values in $[0, 1]$, and let $X := X_1 + \cdots + X_k$. Then, for any $\delta \in [0, 1]$,
$$\Pr[X \geq (1 + \delta)\E[X]] \leq \exp\left(\frac{-\delta^2 \E[X]}{3}\right)$$
and
$$\Pr[X \leq (1 - \delta)\E[X]] \leq \exp\left(\frac{-\delta^2 \E[X]}{2}\right).$$
\end{lemma}

The second statement is a simple upper bound on the expression $1+x$.

\begin{lemma}
\label{lem:1+x}
For every real number $x$, we have $1+x \le e^x$.
\end{lemma}

We end this section with a lemma on the degree of alternatives in a tournament.

\begin{lemma}
\label{lem:degree-sequence}
Let $1\leq r\leq n$.
For any tournament $T$, the average outdegree and the average indegree of the alternatives in any subset of size $r$ are at least $(r-1)/2$.
Similarly, before a tournament is generated from the generalized random model, the average expected outdegree of the alternatives in any subset of size $r$ is at least $(r-1)/2$.
\end{lemma}

\begin{proof}
Given a tournament $T$, in a subset of alternatives of size $r$, there are a total of $r(r-1)/2$ edges.
Hence, the sum of the outdegrees of the $r$ alternatives is at least $r(r-1)/2$, implying that their average is at least $(r-1)/2$.
Likewise, the sum of the indegrees of the $r$ alternatives is at least $r(r-1)/2$, so their average is also at least $(r-1)/2$.

Now, consider a tournament with probabilities $p_{i,j}$ as in the generalized random model, and a subset $B$ of $r$ alternatives.
Each edge between two alternatives in $B$ adds $1$ to the sum of expected outdegree of these two alternatives.
Since there are $r(r-1)/2$ edges between alternatives in $B$, the sum of expected outdegree of the alternatives in $B$ is at least $r(r-1)/2$, and the average expected outdegree is therefore at least $(r-1)/2$.
\end{proof}

Unless a base is explicitly specified, $\log$ refers to the natural logarithm.

\section{Generalized Kings}
\label{sec:generalized-kings}

Recall the result of Saile and Suksompong~\cite{SaileSu20} that in the generalized random model, all alternatives are $2$-kings with high probability only if $p\in\Omega(\sqrt{\log n/n})$.
For our first result, we show that $3$-kings are not as selective: even when $p\in\Omega(\log n/n)$, it is already likely that none of the alternatives is excluded by this set.

\begin{theorem}
\label{thm:3king-positive}
Assume that a tournament $T$ is generated according to the generalized random model, and that
$
p_{i,j} \in \left[30\log n/n, 1 - 30\log n/n\right]
$
for all $i\ne j$.
Then with high probability, all alternatives in $T$ are $3$-kings.
\end{theorem}

To prove this theorem, we establish a rather general lemma on the probability of one set dominating another, which will also be useful in our analysis of $5$-kings later.

\begin{lemma} \label{lem:helper-subsets-dominates}
Let $T_0$ be a tournament with $n_0 \ge 10$ alternatives $V_0 := \{x_1, \dots, x_{n_0}\}$, and let 
$
q_{1,1}, q_{1,2}, \dots, q_{n_0,1}, q_{n_0,2} \in \left[\frac{10 \lambda}{n_0},1\right]
$
for some $1 \le \lambda \le \frac{n_0}{10}$.
Suppose that we randomly create a set $S_1$ by including each alternative $x_i$ independently with probability $q_{i,1}$, and a set $S_2$ by including each alternative $x_i$ independently with probability $q_{i,2}$. 
Then, 
$\Pr[S_1 \cap S_2 = \emptyset \text{ and } S_1 \succ S_2] \leq e^{-\lambda}.$
\end{lemma}

\begin{proof}
It suffices to prove the lemma when $q_{1,1} = q_{1,2} = \dots = q_{n_0,1} = q_{n_0,2} = \frac{10 \lambda}{n_0}$ since increasing these probabilities does not decrease $\Pr[S_1 \cap S_2 = \emptyset \text{ and } S_1 \succ S_2]$.
Observe that if $S_1\succ S_2$, then every alternative in $S_1$ must have a strictly higher outdegree than every alternative in $S_2$ when we restrict the tournament $T_0$ to $S_1\cup S_2$.

Consider any set $S$ of alternatives; let $z_1, \dots, z_t$ be its elements in nondecreasing order of outdegree in the restriction of $T_0$ to $S$. Notice that, when we condition on $S_1 \cap S_2 = \emptyset$ and $S_1 \cup S_2 = S$, since all of our probabilities $q_{i,j}$ are equal, the set $S_1$ has a uniform distribution over all subsets of $S$. 
Furthermore, $S_1 \not\succ S_2$ unless $S_1 = \{z_1, \dots, z_i\}$ for some $i \in \{0, \dots, t\}$. Hence, we have
\begin{align*}
\Pr[S_1 \succ S_2 \mid S_1 \cup S_2 = S \text{ and } S_1 \cap S_2 = \emptyset] \leq \frac{t + 1}{2^t}.
\end{align*}
This implies that
\begin{align*}
\Pr[S_1 \cap S_2 = \emptyset \text{ and } S_1 \succ S_2 \mid S_1 \cup S_2 = S] \leq \frac{t + 1}{2^t}.
\end{align*}
Using the law of total expectation, we can bound the desired probability as follows:
\begin{align*}
\Pr&[S_1 \cap S_2 = \emptyset \text{ and } S_1 \succ S_2] \\
&= \sum_{S \subseteq V_0} \Pr[S_1 \cap S_2 = \emptyset \text{ and } S_1 \succ S_2 \mid S_1 \cup S_2 = S]  \cdot \Pr[S_1 \cup S_2 = S]  \\
&\leq \sum_{S \subseteq V_0}  \frac{|S| + 1}{2^{|S|}} \cdot \Pr[S_1 \cup S_2 = S].
\end{align*}

Note that each alternative is included in $S_1 \cup S_2$ with probability exactly $\gamma := 1 - (1 - \frac{10\lambda}{n_0})^2$. Hence, we have $\Pr[S_1 \cup S_2 = S] = \gamma^{|S|} (1 - \gamma)^{n_0 - |S|}$. Substituting this in the above inequality, we get
\begin{align*}
\Pr&[S_1 \cap S_2 = \emptyset \text{ and } S_1 \succ S_2] \\
&\leq \sum_{S \subseteq V_0}  \frac{|S| + 1}{2^{|S|}} \cdot \gamma^{|S|} (1 - \gamma)^{n_0 - |S|}\\
&= \sum_{t=0}^{n_0} \binom{n_0}{t} \cdot (t + 1) \cdot (0.5\gamma)^t (1 - \gamma)^{n_0 - t} \\
&= \sum_{t=0}^{n_0} \binom{n_0}{t} \cdot t \cdot (0.5\gamma)^t (1 - \gamma)^{n_0 - t}  + \sum_{t=0}^{n_0} \binom{n_0}{t} \cdot (0.5\gamma)^t (1 - \gamma)^{n_0 - t} \\
&= n_0 \sum_{t=1}^{n_0} \binom{n_0 - 1}{t - 1} \cdot (0.5\gamma)^t (1 - \gamma)^{n_0 - t}  + \sum_{t=0}^{n_0} \binom{n_0}{t} \cdot (0.5\gamma)^t (1 - \gamma)^{n_0 - t} \\
&= n_0 (0.5\gamma) (1 - 0.5\gamma)^{n_0 - 1} + (1 - 0.5\gamma)^{n_0},
\end{align*}
where we use the binomial expansion of $(0.5\gamma+(1-\gamma))^{n_0-1}$ and $(0.5\gamma+(1-\gamma))^{n_0}$ for the last equality.

Finally, one can check that $\gamma \in [10\lambda / n_0, 20\lambda/n_0]$.
Therefore, we have
\begin{align*}
\Pr[S_1 \cap S_2 = \emptyset \text{ and } S_1 \succ S_2] 
&\leq n_0 (0.5\gamma) \left(1 - \frac{5 \lambda}{n_0}\right)^{n_0 - 1} + \left(1 - \frac{5 \lambda}{n_0}\right)^{n_0} \\
&\leq n_0 \left(\frac{10\lambda}{n_0}\right) \left(1 - \frac{5 \lambda}{n_0}\right)^{n_0 - 1} + \left(1 - \frac{5 \lambda}{n_0}\right)^{n_0} \\
&\leq 10 \lambda \cdot e^{-4\lambda} + e^{-5\lambda} \\
&\leq e^{-\lambda},
\end{align*}
where in the third inequality we use Lemma~\ref{lem:1+x} and the fact that $n_0 \geq 10$, which follows from our assumption that $1\le \lambda\le n_0/10$, and in the last inequality we use the facts that $\lambda \geq 1$, the function $f(y) := 10y\cdot e^{-3y} + e^{-4y}$ is decreasing for $y\in[1,\infty)$, and $f(1) < 1$.
\end{proof}

Lemma~\ref{lem:helper-subsets-dominates} allows for a short proof of Theorem~\ref{thm:3king-positive}.

\begin{proof}[Proof of Theorem~\ref{thm:3king-positive}]
Fix a pair of distinct alternatives $x_i,x_j$.
We first bound the probability that $x_j$ cannot reach $x_i$ via a directed path of length at most three.

Consider the tournament $T^0$ defined by restricting $T$ to $V^0 := V \setminus \{x_i, x_j\}$. 
Let $S_1$ denote the set of alternatives in $V^0$ that dominate $x_i$ with respect to $T$, and let $S_2$ denote the set of alternatives in $V^0$ that are dominated by $x_j$ with respect to $T$. 
Notice that if $S_1 \cap S_2 \ne \emptyset$ or $S_1 \not\succ S_2$, then there is a path of length at most three from $x_j$ to $x_i$. 
Furthermore, from the assumption of the theorem, each alternative belongs to each of $S_1$ and $S_2$ independently with probability at least $\frac{30 \log n}{n}$, which is at least $\frac{25 \log n}{|V^0|}$ for any $n\ge 12$. 
As a result, we may apply Lemma~\ref{lem:helper-subsets-dominates} with $\lambda = 2.5 \log n$ (note that the size of $T^0$ is $n-2\ge 10$), which gives
\begin{align*}
\Pr[\text{there is } &\text{no path of length at most three from } x_j \text{ to } x_i] \\
&\leq \Pr[S_1 \cap S_2 = \emptyset \text{ and } S_1 \succ S_2] \\
&\leq e^{-\lambda} \\
&= 1 / n^{2.5}.
\end{align*}

Finally, applying the union bound over all (ordered) pairs of alternatives $x_i\ne x_j$, the probability that some alternative cannot reach some other alternative via a directed path of length at most three is no more than $1/n^{0.5}$, which converges to $0$ as $n$ goes to infinity.
\end{proof}

Next, we show that in the Condorcet random model, if $p\in\Theta(\log n/n)$ and the associated constant is low enough, there is likely to be an alternative that is not a $4$-king.
Combined with Theorem~\ref{thm:3king-positive}, this implies that the bound $\Theta(\log n/n)$ is asymptotically tight for both $3$- and $4$-kings in both random models that we consider.

\begin{theorem}
\label{thm:4king-negative}
Assume that a tournament $T$ is generated according to the Condorcet random model, and that 
$
p \le 0.1\log n/n.
$
Then with high probability, there exists an alternative in $T$ that is not a $4$-king.
\end{theorem}

\begin{proof}
Let $r = \left\lceil n^{0.45} \right\rceil$.
We assume that $n \ge 7$; in this range, it holds that $n > 2r$.
Let us partition the set $V = \{x_1, \dots, x_n\}$ into three parts: $V_{\text{top}} := \{x_1, \dots, x_r\}$, $V_{\text{mid}} := \{x_{r + 1}, \dots, x_{n - r}\}$, and $V_{\text{bottom}} := \{x_{n - r + 1}, \dots, x_n\}$. 
Furthermore, we define the following four events:
\begin{itemize}
\item $E_1$: There exists $x_{i^*} \in V_{\text{top}}$ such that $x_{i^*} \succ V_{\text{mid}}$.
\item $E_2$: There exists $x_{j^*} \in V_{\text{bottom}}$ such that $V_{\text{mid}} \succ x_{j^*}$.
\item $E_3$: $V_{\text{top}} \succ V_{\text{bottom}}$.
\item $E_4$: For every $x_\ell\in V_{\text{mid}}$, either $V_{\text{top}}\succ x_\ell$ or $x_\ell\succ V_{\text{bottom}}$.
\end{itemize}
Before we proceed, let us note that if all four events occur, then $x_{j^*}$ is not a $4$-king. 
To see that this is the case, assume for the sake of contradiction that $x_{j^*}$ is a $4$-king; in particular, there exists a directed path of length at most four from $x_{j^*}$ to $x_{i^*}$. 
From $E_2$ and $E_3$, the alternatives dominated by $x_{j^*}$ form a subset of $V_{\text{bottom}}$. 
From $E_1$ and $E_3$, the alternatives dominating $x_{i^*}$ form a subset of $V_{\text{top}}$. 
Since there is a path of length at most four from $x_{j^*}$ to $x_{i^*}$, there must exist $x_j \in V_{\text{bottom}}$ and $x_i \in V_{\text{top}}$ such that $x_j$ can reach $x_i$ via a path of length at most two; however, this is impossible by $E_3$ and $E_4$.

Below, we will argue that each of $E_1, E_2, E_3$, and $E_4$ occurs with high probability. 
Once this is the case, the union bound together with the argument in the previous paragraph completes the proof of Theorem~\ref{thm:4king-negative}.

We start by showing that $E_1$ occurs with high probability. We may write the probability of the complement event $\neg E_1$ as
\begin{align*}
\Pr[\neg E_1] 
&= \Pr[\forall x_i \in V_{\text{top}}, \exists x_j \in V_{\text{mid}}, x_j \succ x_i] \\
&= \prod_{i=1}^r \Pr[\exists x_j \in V_{\text{mid}}, x_j \succ x_i] \\
&= \prod_{i=1}^r \left(1 - \Pr[\forall x_j \in V_{\text{mid}}, x_i \succ x_j]\right) \\
&= \prod_{i=1}^r \left(1 - \prod_{j=r+1}^{n - r} \Pr[x_i \succ x_j]\right) \\
&= \prod_{i=1}^r \left(1 - (1 - p)^{n-2r}\right) \\
&= \left(1 - (1 - p)^{n-2r}\right)^r \\
&\leq \left(1 - (1 - p)^n\right)^r \\
&\leq \left(1 - \frac{1}{(1 + 2p)^n}\right)^r \\
&\leq \left(1 - \frac{1}{e^{2pn}}\right)^r \\
&\leq \left(1 - \frac{1}{n^{0.2}}\right)^r \leq e^{-r / n^{0.2}} \leq e^{-n^{0.1}}.
\end{align*}
The second inequality holds since $(1-p)(1+2p)\ge 1$ for all $p\in [0,1/2]$, while the third and fifth inequalities follow from Lemma~\ref{lem:1+x}.
Thus, $E_1$ occurs with high probability. 
A symmetric argument implies that $E_2$ also occurs with high probability.

Next, consider $E_3$. We may use the union bound to bound $\Pr[\neg E_3]$ as follows:
\begin{align*}
\Pr[\neg E_3] 
&= \Pr[\exists x_i \in V_{\text{top}}, \exists x_j \in V_{\text{bottom}}, x_j \succ x_i] \\
&\leq \sum_{i=1}^r \sum_{j=n-r+1}^n \Pr[x_j \succ x_i] \\
&= \sum_{i=1}^r \sum_{j=n-r+1}^n p \\
&= r^2 p \\
&\in O\left(\frac{\log n}{n^{0.1}}\right),
\end{align*}
which converges to zero as $n \to \infty$. As a result, $E_3$ also occurs with high probability.

Finally, for $E_4$, the union bound implies that
\begin{align*}
\Pr[\neg E_4] 
&= \Pr[\exists x_\ell\in V_{\text{mid}}, \exists x_i \in V_{\text{top}}, \exists x_j \in V_{\text{bottom}},  x_j\succ x_\ell \text{ and } x_\ell \succ x_i] \\
&\leq \sum_{\ell = r+1}^{n-r}\Pr[\exists x_i \in V_{\text{top}}, \exists x_j \in V_{\text{bottom}}, x_j\succ x_\ell \text{ and } x_\ell \succ x_i].
\end{align*}
For each $x_\ell\in V_{\text{mid}}$, the latter probability is 

\begin{align*}
\Pr&[\exists x_i \in V_{\text{top}}, \exists x_j \in V_{\text{bottom}}, x_j\succ x_\ell \text{ and } x_\ell \succ x_i] \\
&= \Pr[\exists x_i \in V_{\text{top}}, x_\ell \succ x_i]\cdot \Pr[\exists x_j \in V_{\text{bottom}}, x_j\succ x_\ell] \\
&\leq \sum_{i=1}^r \Pr[x_\ell\succ x_i]\cdot \sum_{i=n-r+1}^n \Pr[x_j\succ x_\ell] \\
&= (pr)^2,
\end{align*}
where we use independence for the first equality and the union bound for the inequality.
Hence, 
\[
\Pr[\neg E_4] \leq n(pr)^2\in O(\log^2n/n^{0.1}),
\]
which vanishes for large $n$.
This means that $E_4$ occurs with high probability as well, concluding the proof.
\end{proof}

Our results so far demonstrate that $k$-kings for any $k\geq 3$ are much closer to the $(n-1)$-kings than to $2$-kings in terms of discriminative power.
In the remainder of this section, we show that there is virtually no difference in selectiveness between $k$-kings for $k\geq 6$ and $(n-1)$-kings.
We begin by showing that in the Condorcet random model, all alternatives are likely to be $5$-kings provided that $p\in \omega(1/n)$---this gives a complete characterization of the probability threshold for the Condorcet random model.

\begin{theorem}
\label{thm:5king-positive-Condorcet}
Assume that a tournament $T$ is generated according to the Condorcet random model, and that 
$
p \in \omega(1/n).
$
Then with high probability, all alternatives in $T$ are $5$-kings.
\end{theorem}

\begin{proof}
If $p\in\omega(\log n/n)$, the result already follows from Theorem~\ref{thm:3king-positive}, so assume that $p\in O(\log n/n)$.
In particular, $p \in o(1/\sqrt{n})$.
Let $n\ge 5$, $V_{\text{top}} := \{x_1,x_2\}$, and $V_{\text{bottom}} := \{x_{n-1},x_n\}$.
We define the following three events:
\begin{itemize}
\item $E_1$: For every $x_m \not\in V_{\text{top}}$, there exists $x_i\in V_{\text{top}}$ such that $x_i\succ x_m$.
\item $E_2$: For every $x_\ell \not\in V_{\text{bottom}}$, there exists $x_j\in V_{\text{bottom}}$ such that $x_\ell\succ x_j$.
\item $E_3$: For every $x_i\in V_{\text{top}}$ and $x_j\in V_{\text{bottom}}$, there exists a directed path of length at most three from $x_j$ to $x_i$.
\end{itemize}
Before we proceed, let us note that if all three events occur, then all alternatives in $T$ are $5$-kings.
Indeed, consider any pair of distinct alternatives $x_\ell$ and $x_m$.
If $x_\ell\not\in V_{\text{bottom}}$, let $x_j\in V_{\text{bottom}}$ be such that $x_\ell\succ x_j$; the existence of such $j$ is guaranteed by $E_2$.
Else, let $j = \ell$.
Analogously, if $x_m\not\in V_{\text{top}}$, let $x_i\in V_{\text{top}}$ be such that $x_i\succ x_m$; the existence of such $i$ is guaranteed by $E_1$.
Else, let $i=m$.
From $E_3$, $x_j$ can reach $x_i$ via a directed path of length at most three.
This yields a path of length at most five from $x_\ell$ to $x_m$.

By the union bound, it suffices to show that each of $E_1, E_2$, and $E_3$ occurs with high probability. 
For $E_1$, the probability that its complement $\neg E_1$ occurs is
\begin{align*}
\Pr[\neg E_1]
&= \Pr[\exists x_m\not\in V_{\text{top}}, \forall x_i\in V_{\text{top}}, x_m\succ x_i] \\
&\leq \sum_{x_m\not\in V_{\text{top}}} \Pr[\forall x_i\in V_{\text{top}}, x_m\succ x_i] \\
&= \sum_{x_m\not\in V_{\text{top}}} \Pr[x_m\succ x_1]\cdot \Pr[x_m\succ x_2] \\
&= \sum_{x_m\not\in V_{\text{top}}} p^2 \\
&\leq p^2n \\
&\in o(1),
\end{align*}
where the first inequality follows from the union bound and the second equality from independence.
An analogous argument shows that $E_2$ also occurs with high probability.

Finally, consider $E_3$.
Since both $V_{\text{top}}$ and $V_{\text{bottom}}$ are of constant size, by the union bound, it suffices to show that for each fixed pair $x_i\in V_{\text{top}}$ and $x_j\in V_{\text{bottom}}$, a path of length at most three from $x_j$ to $x_i$ exists with high probability.
We show this for $x_1$ and $x_n$; the proofs for other pairs are similar.
Define $V_{\text{top-half}} = \{x_1,\dots,x_{\lfloor n/2\rfloor}\}$ and $V_{\text{bottom-half}} = \{x_{\lceil n/2\rceil + 1},\dots,x_n\}$.
The probability that $V_{\text{top-half}} \succ x_n$ is
\begin{align*}
(1-p)^{\lfloor n/2\rfloor} 
&\leq (1-p)^{n/3} 
\leq e^{-pn/3},
\end{align*}
which converges to $0$ as $n\rightarrow\infty$ since $p\in\omega(1/n)$; here, the second inequality holds by Lemma~\ref{lem:1+x}.
Hence, with high probability, there exists $x_\ell \in V_{\text{top-half}}$ such that $x_n\succ x_\ell$. 
Likewise, with high probability, there exists $x_m\in V_{\text{bottom-half}}$ such that $x_m\succ x_1$.
Since $\ell < m$, the probability that $x_\ell \succ x_m$ is $1-p$, which approaches $1$ for large $n$.
Hence, the union bound again implies that with high probability, $x_n$ can reach $x_1$ via the path $x_n\succ x_\ell\succ x_m\succ x_1$, as desired.
\end{proof}

For the generalized random model, we show that as long as $p\in\omega(1/n)$, all alternatives are likely to be $6$-kings.

\begin{theorem}
\label{thm:6king-positive}
Assume that a tournament $T$ is generated according to the generalized random model, and that 
$
p_{i,j} \in \omega(1/n)
$
for all $i\neq j$.
Then with high probability, all alternatives in $T$ are $6$-kings.
\end{theorem}

To prove Theorem~\ref{thm:6king-positive}, we first establish the following auxiliary lemma.

\begin{lemma}
\label{lem:middle-vertex}
For any tournament $T$, there exists an alternative with both indegree and outdegree larger than $n/4 - 1$.
\end{lemma}

\begin{proof}
Let us order the alternatives of $T$ as $x_{\sigma(1)},\dots,x_{\sigma(n)}$ in nondecreasing order of outdegree, and let $x^* = x_{\sigma(\lfloor n/2\rfloor)}$. By applying Lemma~\ref{lem:degree-sequence} on $x_{\sigma(1)}, \dots, x_{\sigma(\lfloor n/2\rfloor)}$, we can conclude that the outdegree of $x^*$ is at least $(\lfloor n/2\rfloor - 1)/2 > n/4 - 1$. Similarly, applying Lemma~\ref{lem:degree-sequence} on $x_{\sigma(\lfloor n/2\rfloor)}, \dots, x_{\sigma(n)}$ implies that the indegree of $x^*$ is at least $((n - \lfloor n/2\rfloor + 1) - 1) / 2 > n/4 - 1$.
\end{proof}

The high-level idea of the proof of Theorem~\ref{thm:6king-positive} is to identify an alternative $x^*$ as a ``hub'' that can both reach and be reached by every other alternative via a path of length at most three.
The proof below formalizes this idea.

\begin{proof}[Proof of Theorem~\ref{thm:6king-positive}]
Define a tournament $T'$ on alternatives $x_1,\dots,x_n$ so that for each pair $i\neq j$, we have $x_i\succ x_j$ if $p_{i,j}> 1/2$.
Then, our tournament $T$ is generated by reversing the edges of $T'$ so that the edge between $x_i$ and $x_j$ is reversed with probability $q_{i,j} := \min\{p_{i,j},1-p_{i,j}\}\leq 1/2$, independently of other edges.
Note that $q_{i,j}\in \omega(1/n)$.

From Lemma~\ref{lem:middle-vertex}, there exists an alternative $x^*$ whose indegree and outdegree are both larger than $n/4 - 1$, which is at least $0.24n$ for any $n\ge 100$.

We claim that with high probability, in $T$, every alternative has a path of length at most three to $x^*$, and $x^*$ has a path of length at most three to every alternative; this is sufficient to obtain the desired conclusion.
By symmetry, we only need to show the latter claim. 

Let $r=\lceil 0.24n\rceil$, and let $y_1,\dots,y_r$ be alternatives dominated by $x^*$ in $T'$.
For $1\leq i\leq r$, let $Y_i$ be an indicator random variable that indicates whether the edge from $x^*$ to $y_i$ is reversed when generating $T$ from $T'$---in particular, $Y_i$ takes on the value $1$ if the edge is reversed, and $0$ otherwise.
Note that $\E[Y_i] \le 1/2$ for each $i$, and define $Y := \sum_{i=1}^r Y_i$.
Moreover, define $Y_i' := Y_i  + 1/2 - \E[Y_i]$ and $Y' := \sum_{i=1}^r Y_i'$.
We have $\E[Y_i'] = 1/2$ and $\E[Y'] = r/2$.
By Lemma~\ref{lem:chernoff}, it follows that
\begin{align*}
\Pr[Y \ge 7r/12]
&\leq \Pr[Y' \ge 7r/12] \\
&= \Pr\left[Y' \ge \frac{7}{6}\cdot \E[Y'] \right] \\
&\leq \exp\left(-\left(\frac{1}{6}\right)^2\cdot\frac{\E[Y']}{3}\right) = \exp\left(-\frac{r}{216}\right),
\end{align*}
which converges to $0$ for large $n$.
This means that with high probability, we have $Y < 7r/12$; when this happens, the outdegree of $x^*$ in $T$ is at least $5r/12 \geq 0.1n$.
From now on, we condition on the event that $x^*$ has outdegree at least $0.1n$ in $T$. 
Note that this conditioning only involves edges adjacent to $x^*$. 
In particular, for any $i \neq j$ such that $x^*\not\in\{x_i, x_j\}$, we still have $x_i \succ x_j$ independently with probability $p_{i, j}$.

Let $S$ be the set of alternatives dominated by $x^*$ in $T$; by our assumption, $|S| \geq 0.1 n$. 
Since $x^*$ can reach each alternative in $S$ via a single edge, it suffices to show that, with high probability, $x^*$ can reach every alternative in $\bar{S} := \{x_1, \dots, x_n\} \setminus (S \cup \{x^*\})$ via a path of length at most three. 

Consider the tournament $T_{\text{res}}$, which is the restriction of $T$ on the alternatives in $\bar{S}$. 
Let $z_1, \dots, z_{|\bar{S}|}$ be the alternatives in $\bar{S}$ in nondecreasing order of indegree. 
For each $\ell$, let $D_\ell \subseteq \bar{S}$ denote the set of alternatives that dominate $z_\ell$ in $T_{\text{res}}$, and let $\widehat{D}_\ell= D_\ell \cup \{z_\ell\}$. 
From Lemma~\ref{lem:degree-sequence}, we have $|D_\ell|\geq (\ell-1)/2$, and so $|\widehat{D}_\ell| \geq (\ell + 1) / 2$.

Notice that if some $s\in S$ dominates some $d\in \widehat{D}_\ell$, then $x^*$ can reach $z_\ell$ via a path of length at most three, i.e., the path $x^*\succ s\succ z_\ell$ if $d = z_\ell$, or the path $x^*\succ s\succ d\succ z_\ell$ if $d\in D_\ell$.
Hence, if there is no path from $x^*$ to $z_\ell$ of length at most three, we must have $\widehat{D}_\ell \succ S$ in $T$.
This happens with probability
\begin{align*}
\prod_{x_i \in \widehat{D}_\ell} \prod_{x_j \in S} p_{i, j} &\leq (1 - p_{\min})^{|\widehat{D}_\ell| \cdot |S|} \leq (1 - p_{\min})^{0.05 (\ell + 1) n}.
\end{align*}
where $p_{\min} := \min_{i \ne j} p_{i, j} \in \omega(1/n)$.

Hence, by union bound, the probability that there exists an alternative in $\bar{S}$ which cannot be reached from $x^*$ via a path of length at most three is bounded above by
\begin{align*}
\sum_{\ell=1}^{|\bar{S}|} (1 - p_{\min})^{0.05 (\ell + 1) n}
&= \sum_{\ell=1}^{|\bar{S}|} \left((1 - p_{\min})^{0.05n}\right)^{(\ell + 1)} \leq \sum_{\ell=1}^{\infty} \left(e^{-0.05p_{\min} n}\right)^{(\ell + 1)},
\end{align*}
where the inequality holds by Lemma~\ref{lem:1+x}.
Let $\zeta := e^{-0.05p_{\min} n}$. Since $p_{\min} \in \omega(1/n)$, we have $\zeta \in o(1)$. 
This means that the right-hand side expression above is equal to $\frac{\zeta^2}{1 - \zeta} \in o(1)$, which concludes our proof.
\end{proof}

In light of Theorem~\ref{thm:6king-positive}, the only remaining gap in our probability threshold characterization is for $5$-kings in the generalized random model.
We conjecture that the true threshold is $\omega(1/n)$, but our proof of Theorem~\ref{thm:5king-positive-Condorcet} relies on the ordering of the alternatives in the Condorcet random model and cannot be easily extended.
Instead, we present a slightly weaker bound of $\Omega(\log\log n/n)$---this shows that $5$-kings are closer to $6$-kings than to $4$-kings with respect to selective power.
To establish this bound, we need a lemma on generalized dominating sets.

\begin{definition}
Given a positive integer $r$, a set of alternatives $D$ is said to be an \emph{$r$-dominating set} of a tournament $T$ if every alternative $x \notin D$ is dominated by at least $r$ alternatives in $D$. 
\end{definition}

\begin{lemma} \label{lem:gen-domset-size}
For any tournament $T$ and any positive integer $r$, there exists an $r$-dominating set of $T$ of size at most $r \lceil \log_2 n \rceil$.
\end{lemma}

\begin{proof}
Let us start with $D = \emptyset$ and repeat the following procedure $r$ times: find a minimum dominating set $S$ of the restriction of $T$ on $V \setminus D$, and update $D$ to $D \cup S$. 
It is clear that the final set $D$ is an $r$-dominating set of $T$.
Furthermore, it is well-known~\cite[Fact 2.5]{MegiddoVi88} that any tournament has a dominating set of size at most $\lceil \log_2 n\rceil$, which implies that the final set $D$ is of size at most $r \lceil \log_2 n \rceil$.
\end{proof}

We are now ready to establish our result on $5$-kings.

\begin{theorem}
\label{thm:5king-positive-generalized}
Assume that a tournament $T$ is generated according to the generalized random model, and that 
$
p_{i,j} \in \left[50 (\log\log n)/n, 1 - 50 (\log\log n)/n\right]
$
for all $i\neq j$.
Then with high probability, all alternatives in $T$ are $5$-kings.
\end{theorem}

\begin{proof}
Define a tournament $T'$ on alternatives $x_1,\dots,x_n$ so that for each pair $i\neq j$, we have $x_i\succ x_j$ if $p_{i,j}> 1/2$.
Then, our tournament $T$ is generated by reversing the edges of $T'$ so that the edge between $x_i$ and $x_j$ is reversed with probability $q_{i,j} := \min\{p_{i,j},1-p_{i,j}\}\leq 1/2$, independently of other edges.
For each alternative $x$, let $I(x)$ denote the set of alternatives that dominate $x$ in $T'$.

Let $r = \lceil 1.1\log_2 n\rceil$, and let $D$ and $D_{\text{inv}}$ be a minimum $r$-dominating set of the tournament $T'$ and its ``inverse'' constructed by reversing all edges in $T'$, respectively. 
From Lemma~\ref{lem:gen-domset-size}, we have $|D|, |D_{\text{inv}}| \leq r \lceil \log_2 n \rceil \in O((\log n)^2)$.

We define the following three events in $T$:
\begin{itemize}
\item $E_1$: For every $x_m \not\in D$, there exists $x_i\in D$ such that $x_i\succ x_m$.
\item $E_2$: For every $x_\ell \not\in D_{\text{inv}}$, there exists $x_j\in D_{\text{inv}}$ such that $x_\ell\succ x_j$.
\item $E_3$: For every $x_i\in D$ and $x_j\in D_{\text{inv}}$, there exists a directed path of length at most three from $x_j$ to $x_i$.
\end{itemize}
Similarly to the proof of Theorem~\ref{thm:5king-positive-Condorcet}, when $E_1$, $E_2$, and $E_3$ all occur, every alternative is a $5$-king. As a result, it suffices to show that each of the three events occurs with high probability.

For $E_1$, we have
\begin{align*}
\Pr[\neg E_1]
&= \Pr[\exists x_m\not\in D, \forall x_i\in D, x_m\succ x_i] \\
&\leq \sum_{x_m\not\in D} \Pr[\forall x_i\in D, x_m\succ x_i] \\
&\leq \sum_{x_m\not\in D} \Pr[\forall x_i\in D \cap I(x_m), x_m\succ x_i] \\
&= \sum_{x_m\not\in D} \prod_{x_i\in D \cap I(x_m)} \Pr[x_m\succ x_i] \\
&= \sum_{x_m\not\in D} \prod_{x_i\in D \cap I(x_m)} q_{i, m} \\
&\leq \sum_{x_m\not\in D} (1/2)^{|D \cap I(x_m)|} \\
&\leq \sum_{x_m\not\in D} (1/2)^r \leq \sum_{x_m\not\in D} 1/n^{1.1} \in o(1),
\end{align*}
where the first inequality follows the union bound and the fourth inequality from the fact that $D$ is an $r$-dominating set in $T'$.
An analogous argument shows that $E_2$ also occurs with high probability.

Finally, consider $E_3$.
Since both $D$ and $D_{\text{inv}}$ are of size $O((\log n)^2)$, by the union bound, it suffices to show that for each fixed pair $x_i\in D$ and $x_j\in D_{\text{inv}}$, a path of length at most three from $x_j$ to $x_i$ exists with probability $1 - o(1 / (\log n)^{4})$.

To prove this, consider the tournament $T^0$ defined by restricting $T$ to $V^0 := V \setminus \left(D \cup D_{\text{inv}}\right)$. 
Let $S_1$ denote the set of alternatives in $V^0$ that dominate $x_i$ with respect to $T$, and let $S_2$ denote the set of alternatives in $V^0$ that are dominated by $x_j$ with respect to $T$. 
Notice that if $S_1 \cap S_2 \ne \emptyset$ or $S_1 \not\succ S_2$, then there is a path of length at most three from $x_j$ to $x_i$. 
Furthermore, from the assumption of the theorem, each alternative belongs to each of $S_1$ and $S_2$ independently with probability at least $\frac{50 \log \log n}{n}$, which is at least $\frac{45\log \log n}{|V^0|}$ for any $n$ such that $|D|,|D_{\text{inv}}|\le n/20$; this holds when $n\ge 4000$.
As a result, we may apply Lemma~\ref{lem:helper-subsets-dominates} with $\lambda = 4.5 \log \log n$, which gives
\begin{align*}
\Pr[\text{there is } &\text{no path of length at most three from } x_j \text{ to } x_i] \\
&\leq \Pr[S_1 \cap S_2 = \emptyset \text{ and } S_1 \succ S_2] \\
&\leq e^{-\lambda} \\
&= 1 / (\log n)^{4.5} \\
&\in o(1/ (\log n)^4), 
\end{align*}
which concludes our proof.
\end{proof}

\section{Single-Elimination Winners}
\label{sec:SE-winners}

In this section, we turn our attention to single-elimination winners and derive a tight bound of $\Omega(\log n/n)$ for the generalized random model, thereby strengthening the bound $\Omega(\sqrt{\log n/n})$ of Kim et al.~\cite{KimSuVa17} and matching their bound for the Condorcet random model.
As in previous work on this subject, we assume for simplicity that $n=2^r$ for some positive integer $r$, so $r = \log_2 n$.
In order to construct a winning bracket, a useful notion is that of a ``superking'', introduced by Vassilevska Williams~\cite{Vassilevskawilliams10}.

\begin{definition}
Given a tournament $T$, an alternative $x$ is said to be a \emph{superking} if for every alternative $x'$ such that $x'\succ x$, there exist at least $\log_2 n$ alternatives $x''$ such that $x\succ x''$ and $x''\succ x'$.
\end{definition}

\begin{lemma}[\cite{Vassilevskawilliams10}]
\label{lem:superking}
In any tournament, every superking is a single-elimination winner, and its winning bracket can be computed in polynomial time.
\end{lemma}

Before we proceed to our result, let us briefly recap the proofs of the two aforementioned results by Kim et al.~\cite{KimSuVa17}, and explain why they cannot be used to establish our desired strengthening.
In order to show that all alternatives are single-elimination winners with high probability when $p\in\Omega(\sqrt{\log n/n})$, Kim et al.~showed that all alternatives are likely to be superkings in this range; it is not difficult to verify that this condition no longer holds when $p\in o(\sqrt{\log n/n})$.\footnote{To see this, consider the Condorcet random model with $p\in o(\sqrt{\log n/n})$, and let $x$ and $x'$ be the weakest and strongest alternative, respectively.
It is likely that $x'\succ x$.
Moreover, for any other alternative $y$, the probability that $x\succ y$ and $y\succ x'$ simultaneously hold is $p^2 \in o(\log n/n)$, so the expected number of alternatives $x''\not\in\{x,x'\}$ such that $x\succ x''$ and $x''\succ x'$ is $o(\log n)$.
This means that $x$ is unlikely to be a superking.
}
For the $\Omega(\log n/n)$ bound in the Condorcet random model, they argued that the weakest alternative $x$ is likely to dominate one of the top $n/6$ alternatives, and constructed a winning bracket for this latter alternative $y$ among the top half of the alternatives, so that $x$ can play $y$ in the final round and win the single-elimination tournament.
Since there is no clear notion of strength in the generalized random model (indeed, all alternatives may be roughly equally strong, with no linear order), this approach also does not work for our purpose.

At a high level, our proof proceeds by choosing $r = \log_2n$ alternatives that our desired winning alternative $x$ dominates.
In order to ensure that $x$ can play against these alternatives in the $r$ rounds, we partition the $r$ alternatives along with the remaining alternatives into subsets of size $1,2,\dots,2^{r-1}$, so that the $r$ alternatives are superkings in the respective subtournament and can therefore win a single-elimination tournament with respect to their subset by Lemma~\ref{lem:superking}.

\begin{theorem}
\label{thm:single-elimination}
Assume that a tournament $T$ is generated according to the generalized random model, and that 
$
p_{i,j} \in \left[100\log n/n, 1 - 100\log n/n\right]
$
for all $i\neq j$.
Then with high probability, all alternatives in $T$ are single-elimination winners, and a winning bracket of each alternative can be computed in polynomial time.
\end{theorem}

\begin{proof}
By union bound, it suffices to show that for each alternative $x$, with probability $1-o(1/n)$, $x$ is a single-elimination winner and its winning bracket can be computed in polynomial time.

Fix an alternative $x$. First, note that by Lemma~\ref{lem:degree-sequence}, any subset of alternatives of size at least $0.94n-1$ has average expected outdegree at least $(0.94n-2)/2$, which is greater than $0.45n + 1$ for $n > 100$ (since $n$ is assumed to be a power of two in this section, we must have $n\ge 128$).
Hence, there exist at least $0.06n$ alternatives besides $x$ with expected outdegree at least $0.45n + 1$.
Call this set of alternatives $H$.
Let $h := |H| \ge 0.06n$, and let $z_1,\dots,z_h$ be the alternatives in~$H$.

For $1\leq i\leq h$, let $Z_i$ be an indicator random variable that indicates whether $x$ dominates $z_i$ with respect to $T$---in particular, $Z_i$ takes on the value $1$ if $x\succ z_i$, and $0$ otherwise.
Note that $\E[Z_i] \geq 80\log n/n$ for each $i$.
Define $Z := \sum_{i=1}^h Z_i$, so $\E[Z] \geq (100\log n/n) \cdot |H| \ge 6 \log n \geq 3.7r$.
By Lemma~\ref{lem:chernoff}, we have
\begin{align*}
\Pr[Z \leq r] 
&\leq \Pr\left[Z \leq \frac{1}{3.7}\cdot\E[Z] \right] \\
&\leq \exp\left(-\left(\frac{2.7}{3.7}\right)^2\cdot\frac{\E[Z]}{2}\right) \\
&\leq \exp\left(-\left(\frac{2.7}{3.7}\right)^2\cdot 3 \log n\right) \\
&\leq n^{-1.5} \\
&\in o(1/n).
\end{align*}
From now on, we assume that $Z > r$, meaning that $x$ dominates at least $r$ alternatives in $H$ with respect to $T$.
Let $y_1,\dots,y_r$ be $r$ such alternatives.
We will show that with probability $1-o(1/n)$, there exists a partition of $V\setminus\{x\}$ into $r$ sets $V_1,\dots,V_r$ such that for each $1\leq i\leq r$, the set $V_i$ has size $2^{i-1}$ and contains $y_i$, and moreover $y_i$ is a single-elimination winner in the tournament $T$ restricted to $V_i$.
This is sufficient because we can then choose a bracket so that $x$ beats $y_1$ in the first round, $y_2$ in the second round, and so on, until beating $y_r$ in the final round to win the single-elimination tournament.

Consider an alternative $y\in\{y_1,\dots,y_r\}$, and observe that by definition of $H$, for $n > 100$, the expected outdegree of $y$ with respect to alternatives other than $x$ is at least $0.45n$.
Applying Lemma~\ref{lem:chernoff} in a similar manner as above, the probability that the outdegree of $y$ in $T$ upon excluding the edge between $y$ and $x$ is less than $0.4n$ is at most 
\[
\exp\left(-\left(\frac{1}{9}\right)^2\cdot\frac{0.45n}{2}\right)
\leq e^{-0.002n} 
\in o(1/n^2).
\]
By union bound, we have that with probability at least $1-o(1/n)$, every alternative $y_i$ has outdegree at least $0.4n$ in $T$.
Assume from now on that this is the case.

We now construct our desired partition according to the following steps:
\begin{enumerate}
\item For each $1\leq i\leq r$, add $y_i$ to $V_i$.
\item For each $1\leq i\leq r - 3$, add $2^{i-1}-1$ remaining alternatives that $y_i$ dominates in $T$ to $V_i$.
\item For each $r-2\leq i\leq r$, add $\lfloor 0.091n \rfloor-1$ remaining alternatives that $y_i$ dominates in $T$ to $V_i$.
\item For each $r-2\leq i\leq r$, add remaining alternatives to $V_i$ so that $|V_i| = 2^{i-1}$.
\end{enumerate}
To see that this construction procedure is well-defined, observe that in the first three steps, we add a total of
\begin{align*}
r + \sum_{i=1}^{r-3}(2^{i-1}-1) + \sum_{i=r-2}^r(\lfloor 0.091n \rfloor-1) 
&= \sum_{i=1}^{r-3} 2^{i-1} + \sum_{i=r-2}^r \lfloor 0.091n \rfloor \\
&\leq (2^{r-3}-1) + 3\cdot 0.091n \\
&= 0.398n - 1
\end{align*}
alternatives.
Including $x$, the total number of unavailable alternatives at any point during these three steps is therefore at most $0.398n < 0.4n$. Since every $y_i$ has outdegree at least $0.4n$ in $T$, there always exist available alternatives.

It remains to show that each $y_i$ is a superking in $V_i$ with sufficiently high probability.
For $1\leq i\leq r-3$, this holds with probability $1$ because $y_i$ dominates all other alternatives in $V_i$.
For $r-2\leq i\leq r$, by construction, for $n > 100$, $y_i$ dominates at least $0.091n - 2 \geq 0.071n$ alternatives in $V_i$ with respect to $T$; call this set of alternatives $W$.
Let $u$ be any alternative in $V_i$ that dominates $y_i$.
The expected number of alternatives in $W$ that dominate $u$ is at least $(100\log n/n)\cdot (0.071n) = 7.1\log n > 4.9r$.
Applying Lemma~\ref{lem:chernoff} again, the probability that the number of alternatives in $W$ dominating $u$ is less than $r$ is bounded above by
\[
\exp\left(-\left(\frac{3.9}{4.9}\right)^2\cdot \frac{7.1\log n}{2}\right)
\leq n^{-2.2}.
\]
Taking the union bound over all $u$, we find that $y_i$ is a superking with probability $1-o(1/n)$.
A union bound over all $r-2\leq i \leq r$ implies that with this probability, $y_i$ is a superking in $V_i$ for all $i$, and therefore $x$ is a single-elimination winner.

Finally, from Lemma~\ref{lem:superking}, a winning bracket for $y_i$ in the subtournament with alternatives $V_i$ can be found in polynomial time.
Moreover, our partitioning procedure can be implemented efficiently.
Hence, we can compute a winning bracket for $x$ in polynomial time.
\end{proof}

\section{Number of Kings}
\label{sec:number-kings}

So far, we have addressed the question of when each tournament solution is likely to select \emph{all} alternatives, i.e., the case where the solution is decidedly not useful.
In this section, we move beyond this question (which is also the focus of several previous works) and ask when a small or large number of alternatives are chosen.
Our first result shows that even though $p\in \Omega(\sqrt{\log n/n})$ is required for all alternatives to be $2$-kings with high probability \cite{SaileSu20}, the smaller threshold of $p\in \Omega(\log n/n)$ suffices in order for most of the alternatives to be included.

\begin{theorem}
\label{thm:2king-many}
Assume that a tournament $T$ is generated according to the generalized random model, and that
$
p_{i,j} \in \left[50\log n/n, 1 - 50\log n/n\right]
$
for all $i\ne j$.
Then with high probability, at least $0.9n$ alternatives in $T$ are $2$-kings.
\end{theorem}

\begin{proof}
Assume without loss of generality that the alternatives before the tournament~$T$ is generated are $x_1,\dots,x_n$ in nonincreasing order of expected outdegree. Let $r = \lceil 0.9n \rceil$, and consider any $1\leq i\leq r$. We will show that the probability that $x_i$ is a $2$-king in $T$ is at least $1 - o(1/n)$. The union bound then implies that with high probability, at least $0.9n$ alternatives in $T$ are $2$-kings.

To show that $\Pr[x_i \text{ is a } 2\text{-king in } T]$ is at least $1 - o(1/n)$, the union bound again implies that it suffices to prove that the probability that there is no path of length at most two from $x_i$  to  $x_j$ is at most $o(1/n^2)$ for each alternative $x_j \ne x_i$.

We henceforth fix $1\le i\le r$ and $x_j \ne x_i$. By Lemma~\ref{lem:degree-sequence} on $x_i,x_{i+1}, \dots, x_n$, the expected outdegree of $x_i$ is at least $(n - i) / 2$, which is at least $0.045 n$ for any $n\ge 100$. As a result, for this range of $n$, we have
\begin{align*}
\Pr&[\text{there is no path of length at most two from } x_i \text{ to } x_j] \\
&= \Pr[x_j \succ x_i]\cdot \Pr[\forall k \notin \{i,j\}, x_k \succ x_i \text{ or } x_j \succ x_k] \\
&= (1 - p_{i,j})\cdot \prod_{k \notin \{i,j\}} \Pr[x_k \succ x_i \text{ or } x_j \succ x_k] \\
&= (1 - p_{i,j})\cdot \prod_{k \notin \{i,j\}} \left(1 - p_{i,k} p_{k,j}\right) \\
&\leq \exp(-p_{i,j})\cdot \prod_{k \notin \{i,j\}} \exp(- p_{i,k} p_{k,j}) \\
&\leq \exp\left(-\frac{50\log n}{n} \cdot p_{i,j}\right)\cdot \prod_{k \notin \{i,j\}} \exp\left(-\frac{50\log n}{n} p_{i,k}\right) \\
&= \exp\left(-\frac{50\log n}{n} \left(\sum_{k \neq i} p_{i,k}\right)\right) \\
&\leq \exp\left(-\frac{50\log n}{n} \cdot 0.045 n\right) \\
&= 1/n^{2.25} \\
&\in o(1/n^2).
\end{align*}
We use Lemma~\ref{lem:1+x} for the first inequality, and the assumption that the expected outdegree of $x_i$ is at least $0.045n$ for the last inequality. This concludes our proof.
\end{proof}

Our next result establishes the asymptotic tightness of the threshold in Theorem~\ref{thm:2king-many}, and implies that the transition from a small to a large number of $2$-kings occurs at $p\in\Theta(\log n/n)$ for both random models.

\begin{theorem}
\label{thm:2king-few}
Assume that a tournament $T$ is generated according to the Condorcet random model, and that 
$
p \le 0.1\log n/n.
$
Then with high probability, at most $\sqrt{n} \log n$ alternatives in $T$ are $2$-kings.
\end{theorem}

\begin{proof}
Let $n\ge 4$ and $r = \left\lceil \sqrt{n} \right\rceil$. 
Let us partition the set $V = \{x_1, \dots, x_n\}$ into two parts: $V_{\text{top}} := \{x_1, \dots, x_r\}$ and $V_{\text{bottom}} := \{x_{r + 1}, \dots, x_n\}$. Furthermore, let $S$ denote the set of all alternatives in $V_{\text{bottom}}$ that dominate at least one alternative in $V_{\text{top}}$.
We define the following two events:
\begin{itemize}
\item $E_1$: There exists $x_{i^*} \in V_{\text{top}}$ such that $x_{i^*} \succ V_{\text{bottom}}$.
\item $E_2$: $|S| \leq 0.2 r \log n$.
\end{itemize}
Before we proceed, let us note that if both events occur, there are at most $\sqrt{n} \log n$ $2$-kings. To see this, observe that when $E_1$ occurs, an alternative $x \in V_{\text{bottom}}$ can be a $2$-king only if $x \in S$. As a result, when both events occur, the number of $2$-kings is at most $r + |S|$, which is at most $\sqrt{n} \log n$ for any sufficiently large $n$.
We now argue that each of $E_1, E_2$ occurs with high probability; the union bound then completes the proof.

We start by showing that $E_1$ occurs with high probability. We write the probability of the complement event $\neg E_1$ as
\begin{align*}
\Pr[\neg E_1] 
&= \Pr[\forall x_i \in V_{\text{top}}, \exists x_j \in V_{\text{bottom}}, x_j \succ x_i] \\
&= \prod_{i=1}^r \Pr[\exists x_j \in V_{\text{bottom}}, x_j \succ x_i] \\
&= \prod_{i=1}^r \left(1 - \Pr[\forall x_j \in V_{\text{bottom}}, x_i \succ x_j]\right) \\
&= \prod_{i=1}^r \left(1 - \prod_{j=r+1}^n \Pr[x_i \succ x_j]\right) \\
&= \prod_{i=1}^r \left(1 - (1 - p)^{n-r}\right) \\
&= \left(1 - (1 - p)^{n-r}\right)^r \\
&\leq \left(1 - (1 - p)^n\right)^r \\
&\leq \left(1 - \frac{1}{(1 + 2p)^n}\right)^r \\
&\leq \left(1 - \frac{1}{e^{2pn}}\right)^r \\
&\leq \left(1 - \frac{1}{n^{0.2}}\right)^r \leq e^{-r / n^{0.2}} \leq e^{-n^{0.3}}.
\end{align*}
The second inequality holds since $(1-p)(1+2p)\ge 1$ for all $p\in [0,1/2]$, while the third and fifth inequalities follow from Lemma~\ref{lem:1+x}.
Thus, $E_1$ occurs with high probability. 

Next, we consider $E_2$. Let $P := \{(x_i, x_j) \in V_{\text{top}} \times V_{\text{bottom}} \mid x_j \succ x_i\}$. Notice that each $(x_i, x_j) \in V_{\text{top}} \times V_{\text{bottom}}$ belongs to $P$ independently with probability $p \leq 0.1 \log n / n$. We will bound the probability that $|P| \geq 0.2 r \log n$. This probability does not increase when $p$ decreases, so it suffices to consider $p = 0.1\log n / n$. For this $p$, we have $\E[|P|] = p r (n - r) \leq 0.1r \log n$, and moreover $pr(n - r) \geq 0.01r \log n$ for any $n \geq 4$. As a result, by Lemma~\ref{lem:chernoff}, we have
\begin{align*}
\Pr\left[|P| \geq 0.2 r \log n\right]
&\leq \Pr\left[|P| \geq 2\E[|P|]\right] \\
&\leq \exp\left(-\frac{\E[P]}{3}\right) \\
&\leq \exp\left(-\frac{0.01r \log n}{3}\right) \in o(1).
\end{align*}
Finally, notice that $|S| \leq |P|$, which implies that $E_2$ happens with high probability, as desired. 
\end{proof}

Finally, we show that as long as $p\in\omega(1/n)$, a large number of alternatives are already likely to be $3$-kings (and therefore $k$-kings for every $k\ge 3$).
This bound is again tight: when $p\in O(1/n)$ and the tournament is generated from the Condorcet random model, there is at least a constant probability that the strongest alternative dominates all remaining alternatives (in which case it is the only $k$-king for each $k\ge 2$).
In particular, for each $k\ge 3$, the transition from a small to a large number of $k$-kings occurs at $p\in\Theta(1/n)$ for both random models.

\begin{theorem}
\label{thm:3king-many}
Assume that a tournament $T$ is generated according to the generalized random model, and that $p_{i,j} \in \omega(1/n)$ for all $i\neq j$. Then with high probability, at least $0.9n$ alternatives in $T$ are $3$-kings.
\end{theorem}

\begin{proof}
Assume without loss of generality that the alternatives before the tournament~$T$ is generated are $x_1,\dots,x_n$ in nonincreasing order of expected outdegree. 
Let $p = \min_{i \ne j} p_{i, j}$; by our assumption, $p \in \omega(1/n)$.
Furthermore, let $r = \lceil 0.95n \rceil$, and consider any $1\leq i\leq r$. 

We will show that the probability that $x_i$ is \emph{not} a $3$-king in $T$ is at most $o(1)$. Let $q_i$ denote this probability, and let $q = \max_{i=1}^r q_i$. 
By linearity of expectation, for any $n\ge 100$, we have
\begin{align*}
\E[|\{i \mid 1 \leq i \leq r \text{ and } x_i \text{ is not a $3$-king}\}|] \le 0.96nq.
\end{align*}
As a result, by Markov's inequality, we can derive
\begin{align*}
\Pr&[\text{there are less than } 0.9n \text{ $3$-kings}] \\ &\le \Pr[|\{i \mid 1 \leq i \leq r \text{ and } x_i \text{ is not a $3$-king}\}| \geq 0.05 n] \\
&\le \frac{0.96nq}{0.05n} = 96q/5 \in o(1),
\end{align*}
as desired.

It remains to show that for each $1\le i\le r$, the probability that $x_i$ is not a $3$-king is $o(1)$. By Lemma~\ref{lem:degree-sequence} on $x_i,x_{i+1}, \dots, x_n$, the expected outdegree of $x_i$ is at least $(n - i) / 2$, which is at least $0.02 n$ for any $n\ge 100$. Let $U$ denote the set of alternatives dominated by $x_i$ in $T$. Note that $|U|$ is a sum of $n - 1$ independent Bernoulli random variables, and $\E[|U|] \geq 0.02 n$. Applying Lemma~\ref{lem:chernoff}, we have
\begin{align*}
\Pr[|U| < 0.01n] \leq \exp\left(-\frac{1}{4} \cdot \frac{0.02n}{2}\right) \in o(1).
\end{align*}
Thus, we may henceforth assume that $|U| \geq 0.01n$.

Now, let $T_i$ denote the tournament $T$ restricted to $V \setminus (U \cup \{x_i\})$. Let us order the alternatives in $T_i$ as $x_{\psi(1)}, \dots, x_{\psi(L)}$ in nondecreasing order of indegree in $T_i$, where $L = n - |U| - 1$. For each $j = 1, \dots, L$, let $D_j$ denote the set of alternatives in $T_i$ that dominate $x_{\psi(j)}$. We have
\begin{align*}
\Pr[&\text{there is no path of length at most three from } x_i \text{ to } x_{\psi(j)}] \\
&\leq \Pr[D_j \cup \{x_j\} \succ U] \\
&= \prod_{x_k \in D_j \cup \{x_j\}} \prod_{x_\ell \in U}\Pr[x_k \succ x_\ell] \\
&= \prod_{x_k \in D_j \cup \{x_j\}} \prod_{x_\ell \in U} (1 - p_{\ell k}) \\
&\leq \prod_{x_k \in D_j \cup \{x_j\}} \prod_{x_\ell \in U} (1 - p) \\
&= (1 - p)^{(|D_j| + 1)|U|} \leq e^{-0.02np(|D_j| + 1)} \leq e^{-0.01np(j + 1)},
\end{align*}
where in the last two inequalities we use Lemma~\ref{lem:1+x} and the fact that $|D_j| \geq (j -  1)/2$, which follows from Lemma~\ref{lem:degree-sequence}, respectively.

Hence, by the union bound, the probability that some alternative in $T_i$ is unreachable from $x_i$ in at most three steps is no more than
\begin{align*}
\sum_{j=1}^L e^{-0.01np(j + 1)} \leq \frac{e^{-0.02np}}{1 - e^{0.01np}},
\end{align*}
which is $o(1)$ due to our assumption that $np = \omega(1)$.
As a result, the probability that $x_i$ is not a $3$-king is $o(1)$, as claimed.
\end{proof}

\section{Experiments}
\label{sec:experiments}

In this section, we present experimental results on the behavior of $k$-kings in random tournaments.
Besides the Condorcet random model, we also consider an arguably more realistic model introduced by Saile and Suksompong~\cite{SaileSu20} called the ``gap model''.

\subsection{Condorcet Random Model}

We start with the Condorcet random model.
For each $n\in\{10,20,40,60,80,100\}$, each $p\in\{0, 0.02, 0.04, 0.06, \dots, 0.5\}$, and each $k\in\{2,3,4,n-1\}$, we generated 10000 tournaments according to the Condorcet random model with parameters $n$ and $p$, and took the average of the number of $k$-kings in the generated tournaments.
In order to determine the number of $k$-kings, a simple iterative algorithm running in time $O(n^2)$ suffices for the case $k = n-1$ \cite[p.~71]{BrandtBrHa16}, while an efficient algorithm based on matrix multiplication can be used for the cases $k=2,3,4$ \cite[p.~68]{BrandtBrHa16}.\footnote{The referenced book chapter only gives an algorithm for the case $k=2$, but a similar approach works for $k=3$ and $k=4$.}
The results are shown in Figure~\ref{fig:condorcet}; for ease of comparison, we illustrate the \emph{percentage} (as opposed to the \emph{number}) of $k$-kings among all alternatives in the tournament.\footnote{All averages presented in this paper have standard error within $\pm 0.5\%$.}

\begin{figure}
	\centering
	\begin{subfigure}{.5\textwidth} 
		\centering
\begin{tikzpicture}[scale=0.85]
\begin{axis}[
    title={(a) $n = 10$},
    xlabel={},
    ylabel={},
    xmin=0, xmax=0.5,
    ymin=0, ymax=100,
    xtick= {0,0.1,0.2,0.3,0.4,0.5},
    ytick={0,20,40,60,80,100},
    legend pos=south east,
    ymajorgrids=true,
		xmajorgrids=true,
    grid style=dashed,
]

\addplot[
    color=black,
	  mark size=3,
    mark=triangle*,
  	fill opacity=0.8,
  	draw opacity=0.8,
    ]
    coordinates {
    (0.0,10.0)(0.02,13.265)(0.04,16.634)(0.06,20.168)(0.08,23.493000000000002)(0.1,27.094)(0.12,30.43)(0.14,34.159000000000006)(0.16,37.126)(0.18,40.765)(0.2,43.894)(0.22,47.316)(0.24,50.66799999999999)(0.26,52.720000000000006)(0.28,56.104000000000006)(0.3,58.96)(0.32,61.527)(0.34,63.403999999999996)(0.36,65.32300000000001)(0.38,67.182)(0.4,69.075)(0.42,69.79499999999999)(0.44,71.35600000000001)(0.46,71.804)(0.48,72.61999999999999)(0.5,72.413)
    };
    \addlegendentry{$k=2$}

\addplot[
	    color=red,
	    mark=square*,
			fill opacity=0.8,
  	  draw opacity=0.8,
	    ]
	    coordinates {
	    (0.0,10.0)(0.02,18.325)(0.04,26.413999999999998)(0.06,33.297)(0.08,39.99)(0.1,46.608999999999995)(0.12,52.827999999999996)(0.14,58.772000000000006)(0.16,64.21000000000001)(0.18,69.078)(0.2,73.39699999999999)(0.22,77.435)(0.24,81.52)(0.26,83.69899999999998)(0.28,86.644)(0.3,88.73799999999999)(0.32,90.83000000000001)(0.34,92.00699999999999)(0.36,93.21199999999999)(0.38,94.378)(0.4,95.083)(0.42,95.833)(0.44,96.289)(0.46,96.57199999999999)(0.48,96.40700000000001)(0.5,96.44399999999999)
	    };
	    \addlegendentry{$k=3$}

\addplot[
    color=blue,
    mark=otimes*,
		fill opacity=0.8,
  	draw opacity=0.8,
    ]
    coordinates {
    (0.0,10.0)(0.02,18.832)(0.04,28.267000000000003)(0.06,37.582)(0.08,45.37499999999999)(0.1,53.473)(0.12,60.534000000000006)(0.14,66.88)(0.16,72.527)(0.18,76.026)(0.2,80.323)(0.22,83.353)(0.24,86.464)(0.26,88.833)(0.28,90.946)(0.3,92.461)(0.32,93.85600000000001)(0.34,95.26500000000001)(0.36,95.47900000000001)(0.38,96.409)(0.4,96.892)(0.42,97.13699999999999)(0.44,97.50800000000001)(0.46,98.03200000000001)(0.48,97.941)(0.5,98.042)
    };
    \addlegendentry{$k=4$}
    
\addplot[
    color=green,
    mark=diamond*,
		fill opacity=0.8,
  	draw opacity=0.8,
    ]
    coordinates {
    (0.0,10.0)(0.02,19.584)(0.04,29.434999999999995)(0.06,38.033)(0.08,47.284)(0.1,55.73199999999999)(0.12,62.898)(0.14,68.672)(0.16,73.612)(0.18,77.485)(0.2,81.82300000000001)(0.22,84.64000000000001)(0.24,87.40299999999999)(0.26,89.821)(0.28,91.779)(0.3,92.956)(0.32,94.09299999999999)(0.34,95.062)(0.36,95.931)(0.38,96.58099999999999)(0.4,96.97800000000001)(0.42,97.38899999999998)(0.44,97.523)(0.46,98.02199999999999)(0.48,97.981)(0.5,98.067)
    };
    \addlegendentry{$k=n-1$}
\end{axis}
\end{tikzpicture}
\end{subfigure}\begin{subfigure}{.5\textwidth} 
	\centering
\begin{tikzpicture}[scale=0.85]
\begin{axis}[
    title={(b) $n = 20$},
    xlabel={},
    ylabel={},
    xmin=0, xmax=0.5,
    ymin=0, ymax=100,
    xtick= {0,0.1,0.2,0.3,0.4,0.5},
    ytick={0,20,40,60,80,100},
    legend pos=south east,
    ymajorgrids=true,
		xmajorgrids=true,
    grid style=dashed,
]

\addplot[
    color=black,
    mark size=3,
    mark=triangle*,
  	fill opacity=0.8,
  	draw opacity=0.8,
    ]
    coordinates {
    (0.0,5.0)(0.02,8.882)(0.04,13.4345)(0.06,18.641)(0.08,24.655)(0.1,30.994999999999997)(0.12,37.105000000000004)(0.14,43.3565)(0.16,49.596999999999994)(0.18,55.7485)(0.2,60.9345)(0.22,65.94800000000001)(0.24,70.4955)(0.26,75.0725)(0.28,78.53)(0.3,81.9695)(0.32,85.04350000000001)(0.34,87.53450000000001)(0.36,89.604)(0.38,91.19500000000001)(0.4,92.56750000000001)(0.42,93.74050000000001)(0.44,94.487)(0.46,95.091)(0.48,95.42949999999999)(0.5,95.547)
    };
    \addlegendentry{$k=2$}

\addplot[
	    color=red,
        mark=square*,
        fill opacity=0.8,
        draw opacity=0.8,
	    ]
	    coordinates {
	    (0.0,5.0)(0.02,21.823999999999998)(0.04,37.4345)(0.06,51.078)(0.08,63.229)(0.1,73.2085)(0.12,81.47600000000001)(0.14,87.381)(0.16,91.47)(0.18,94.586)(0.2,96.8005)(0.22,97.88050000000001)(0.24,98.6485)(0.26,99.275)(0.28,99.5225)(0.3,99.656)(0.32,99.87199999999999)(0.34,99.8945)(0.36,99.92749999999998)(0.38,99.976)(0.4,99.99100000000001)(0.42,99.9785)(0.44,99.98849999999999)(0.46,100.0)(0.48,99.9795)(0.5,99.9985)
    	   };
	    \addlegendentry{$k=3$}

\addplot[
    color=blue,
    mark=otimes*,
    fill opacity=0.8,
  	draw opacity=0.8,
    ]
    coordinates {
     (0.0,5.0)(0.02,26.117499999999996)(0.04,45.2775)(0.06,61.9015)(0.08,74.079)(0.1,82.72200000000001)(0.12,88.708)(0.14,92.73999999999998)(0.16,95.554)(0.18,96.892)(0.2,98.0055)(0.22,98.6985)(0.24,99.1845)(0.26,99.59949999999999)(0.28,99.73850000000002)(0.3,99.821)(0.32,99.84299999999999)(0.34,99.9035)(0.36,99.938)(0.38,99.97049999999999)(0.4,99.9995)(0.42,99.962)(0.44,99.999)(0.46,99.9995)(0.48,99.99050000000001)(0.5,99.9995)
    };
    \addlegendentry{$k=4$}
    
\addplot[
    color=green,
    mark=diamond*,
		fill opacity=0.8,
  	draw opacity=0.8,
    ]
    coordinates {
    (0.0,5.0)(0.02,29.024)(0.04,51.0865)(0.06,65.9635)(0.08,76.8885)(0.1,84.59899999999999)(0.12,90.11100000000002)(0.14,93.41799999999999)(0.16,95.165)(0.18,96.938)(0.2,98.0455)(0.22,98.712)(0.24,99.22999999999999)(0.26,99.52450000000002)(0.28,99.719)(0.3,99.831)(0.32,99.84349999999999)(0.34,99.92749999999998)(0.36,99.9405)(0.38,99.9795)(0.4,99.99000000000001)(0.42,99.9995)(0.44,100.0)(0.46,100.0)(0.48,100.0)(0.5,100.0)
    };
    \addlegendentry{$k=n-1$}
\end{axis}
\end{tikzpicture}
\end{subfigure}

\vspace{2mm}

	\centering
	\begin{subfigure}{.5\textwidth} 
		\centering
\begin{tikzpicture}[scale=0.85]
\begin{axis}[
    title={(c) $n = 40$},
    xlabel={},
    ylabel={},
    xmin=0, xmax=0.5,
    ymin=0, ymax=100,
    xtick= {0,0.1,0.2,0.3,0.4,0.5},
    ytick={0,20,40,60,80,100},
    legend pos=south east,
    ymajorgrids=true,
		xmajorgrids=true,
    grid style=dashed,
]

\addplot[
    color=black,
	  mark size=3,
    mark=triangle*,
  	fill opacity=0.8,
  	draw opacity=0.8,
    ]
    coordinates {
    (0.0,2.5)(0.02,7.3862499999999995)(0.04,14.917000000000002)(0.06,25.252)(0.08,36.514)(0.1,47.60525)(0.12,57.35324999999999)(0.14,66.0945)(0.16,73.67225)(0.18,79.779)(0.2,84.846)(0.22,89.0345)(0.24,92.3395)(0.26,94.64925)(0.28,96.51624999999999)(0.3,97.68599999999999)(0.32,98.54175000000001)(0.34,99.09649999999999)(0.36,99.4315)(0.38,99.67175)(0.4,99.79775)(0.42,99.88075)(0.44,99.92775)(0.46,99.95775)(0.48,99.96124999999999)(0.5,99.96650000000001)
    };
    \addlegendentry{$k=2$}

\addplot[
	    color=red,
        mark=square*,
        fill opacity=0.8,
        draw opacity=0.8,
	    ]
	    coordinates {
    (0.0,2.5)(0.02,35.027)(0.04,61.94625)(0.06,81.0895)(0.08,90.97149999999999)(0.1,96.38925)(0.12,98.58700000000002)(0.14,99.4975)(0.16,99.7995)(0.18,99.92974999999998)(0.2,99.96700000000001)(0.22,99.96874999999999)(0.24,99.99924999999999)(0.26,100.0)(0.28,100.0)(0.3,100.0)(0.32,100.0)(0.34,100.0)(0.36,100.0)(0.38,100.0)(0.4,100.0)(0.42,100.0)(0.44,100.0)(0.46,100.0)(0.48,100.0)(0.5,100.0)
	    };
	    \addlegendentry{$k=3$}

\addplot[
    color=blue,
    mark=otimes*,
    fill opacity=0.8,
  	draw opacity=0.8,
    ]
    coordinates {
    (0.0,2.5)(0.02,45.407)(0.04,75.952)(0.06,89.7345)(0.08,95.63275000000002)(0.1,98.2005)(0.12,99.2235)(0.14,99.72775)(0.16,99.83124999999998)(0.18,99.94075000000001)(0.2,99.96924999999999)(0.22,99.99000000000001)(0.24,100.0)(0.26,99.99974999999999)(0.28,100.0)(0.3,100.0)(0.32,100.0)(0.34,100.0)(0.36,100.0)(0.38,100.0)(0.4,100.0)(0.42,100.0)(0.44,100.0)(0.46,100.0)(0.48,100.0)(0.5,100.0)
    };
    \addlegendentry{$k=4$}
\addplot[
    color=green,
    mark=diamond*,
		fill opacity=0.8,
  	draw opacity=0.8,
    ]
    coordinates {
    (0.0,2.5)(0.02,53.2855)(0.04,78.62225)(0.06,90.84225)(0.08,95.8745)(0.1,98.17675)(0.12,99.23724999999999)(0.14,99.71875000000001)(0.16,99.87925)(0.18,99.92125)(0.2,99.99000000000001)(0.22,99.99000000000001)(0.24,99.99024999999999)(0.26,100.0)(0.28,100.0)(0.3,100.0)(0.32,100.0)(0.34,100.0)(0.36,100.0)(0.38,100.0)(0.4,100.0)(0.42,100.0)(0.44,100.0)(0.46,100.0)(0.48,100.0)(0.5,100.0)
    };
    \addlegendentry{$k=n-1$}
\end{axis}
\end{tikzpicture}
\end{subfigure}\begin{subfigure}{.5\textwidth} 
	\centering
\begin{tikzpicture}[scale=0.85]
\begin{axis}[
    title={(d) $n = 60$},
    xlabel={},
    ylabel={},
    xmin=0, xmax=0.5,
    ymin=0, ymax=100,
    xtick= {0,0.1,0.2,0.3,0.4,0.5},
    ytick={0,20,40,60,80,100},
    legend pos=south east,
    ymajorgrids=true,
	  xmajorgrids=true,
    grid style=dashed,
]

\addplot[
    color=black,
    mark size=3,
    mark=triangle*,
  	fill opacity=0.8,
  	draw opacity=0.8,
    ]
    coordinates {
    (0.0,1.6666666666666667)(0.02,7.470333333333333)(0.04,19.250166666666665)(0.06,34.91833333333333)(0.08,50.040499999999994)(0.1,62.2445)(0.12,72.59333333333332)(0.14,80.53416666666666)(0.16,86.53483333333334)(0.18,91.10333333333332)(0.2,94.4095)(0.22,96.60916666666667)(0.24,98.077)(0.26,98.946)(0.28,99.47633333333333)(0.3,99.735)(0.32,99.87283333333333)(0.34,99.9445)(0.36,99.97383333333332)(0.38,99.98933333333333)(0.4,99.99566666666666)(0.42,99.99783333333333)(0.44,99.99900000000001)(0.46,99.99949999999998)(0.48,99.99983333333333)(0.5,99.99983333333333)
    };
    \addlegendentry{$k=2$}

\addplot[
	    color=red,
        mark=square*,
        fill opacity=0.8,
        draw opacity=0.8,
	    ]
	    coordinates {
		(0.0,1.6666666666666667)(0.02,49.317833333333326)(0.04,80.37349999999999)(0.06,93.76366666666668)(0.08,98.33433333333332)(0.1,99.65)(0.12,99.94266666666667)(0.14,99.98433333333334)(0.16,99.99849999999999)(0.18,100.0)(0.2,99.99983333333333)(0.22,100.0)(0.24,100.0)(0.26,100.0)(0.28,100.0)(0.3,100.0)(0.32,100.0)(0.34,100.0)(0.36,100.0)(0.38,100.0)(0.4,100.0)(0.42,100.0)(0.44,100.0)(0.46,100.0)(0.48,100.0)(0.5,100.0)
		};
	    \addlegendentry{$k=3$}

\addplot[
    color=blue,
    mark=otimes*,
    fill opacity=0.8,
  	draw opacity=0.8,
    ]
    coordinates {
    (0.0,1.6666666666666667)(0.02,63.59316666666667)(0.04,90.10066666666667)(0.06,96.86266666666667)(0.08,99.05716666666667)(0.1,99.74133333333333)(0.12,99.94033333333333)(0.14,99.95983333333334)(0.16,99.99)(0.18,100.0)(0.2,100.0)(0.22,100.0)(0.24,100.0)(0.26,100.0)(0.28,100.0)(0.3,100.0)(0.32,100.0)(0.34,100.0)(0.36,100.0)(0.38,100.0)(0.4,100.0)(0.42,100.0)(0.44,100.0)(0.46,100.0)(0.48,100.0)(0.5,100.0)
    };
    \addlegendentry{$k=4$}
\addplot[
    color=green,
    mark=diamond*,
		fill opacity=0.8,
  	draw opacity=0.8,
    ]
    coordinates {
    (0.0,1.6666666666666667)(0.02,68.43033333333332)(0.04,91.07783333333333)(0.06,97.06683333333334)(0.08,99.33583333333334)(0.1,99.7895)(0.12,99.9595)(0.14,99.98016666666668)(0.16,99.98033333333333)(0.18,99.99016666666667)(0.2,99.99983333333333)(0.22,100.0)(0.24,100.0)(0.26,100.0)(0.28,100.0)(0.3,100.0)(0.32,100.0)(0.34,100.0)(0.36,100.0)(0.38,100.0)(0.4,100.0)(0.42,100.0)(0.44,100.0)(0.46,100.0)(0.48,100.0)(0.5,100.0)
    };
    \addlegendentry{$k=n-1$}
\end{axis}
\end{tikzpicture}
\end{subfigure}

\vspace{2mm}

	\begin{subfigure}{.5\textwidth} 
		\centering
\begin{tikzpicture}[scale=0.85]
\begin{axis}[
    title={(e) $n = 80$},
    xlabel={},
    ylabel={},
    xmin=0, xmax=0.5,
    ymin=0, ymax=100,
    xtick= {0,0.1,0.2,0.3,0.4,0.5},
    ytick={0,20,40,60,80,100},
    legend pos=south east,
    ymajorgrids=true,
		xmajorgrids=true,
    grid style=dashed,
]

\addplot[
    color=black,
    mark size=3,
    mark=triangle*,
  	fill opacity=0.8,
  	draw opacity=0.8,
    ]
    coordinates {
    (0.0,1.25)(0.02,8.55575)(0.04,25.281250000000004)(0.06,45.0295)(0.08,61.210125000000005)(0.1,73.228875)(0.12,82.191625)(0.14,88.477125)(0.16,92.86337499999999)(0.18,95.954)(0.2,97.85774999999998)(0.22,98.922125)(0.24,99.51174999999999)(0.26,99.79499999999999)(0.28,99.91287500000001)(0.3,99.96712499999998)(0.32,99.99137499999999)(0.34,99.995125)(0.36,99.99875)(0.38,99.99962500000001)(0.4,100.0)(0.42,100.0)(0.44,100.0)(0.46,100.0)(0.48,100.0)(0.5,100.0)	
    };
    \addlegendentry{$k=2$}

\addplot[
        color=red,
        mark=square*,
        fill opacity=0.8,
        draw opacity=0.8,
	    ]
	    coordinates {
    (0.0,1.25)(0.02,60.7365)(0.04,90.369)(0.06,98.345875)(0.08,99.72787500000001)(0.1,99.964875)(0.12,99.999125)(0.14,100.0)(0.16,100.0)(0.18,100.0)(0.2,100.0)(0.22,100.0)(0.24,100.0)(0.26,100.0)(0.28,100.0)(0.3,100.0)(0.32,100.0)(0.34,100.0)(0.36,100.0)(0.38,100.0)(0.4,100.0)(0.42,100.0)(0.44,100.0)(0.46,100.0)(0.48,100.0)(0.5,100.0)	
		};
	    \addlegendentry{$k=3$}

\addplot[
    color=blue,
    mark=otimes*,
    fill opacity=0.8,
  	draw opacity=0.8,
    ]
    coordinates {
    (0.0,1.25)(0.02,76.23325)(0.04,95.61175)(0.06,98.98425)(0.08,99.85925)(0.1,99.94037499999999)(0.12,99.99974999999999)(0.14,100.0)(0.16,100.0)(0.18,100.0)(0.2,100.0)(0.22,100.0)(0.24,100.0)(0.26,100.0)(0.28,100.0)(0.3,100.0)(0.32,100.0)(0.34,100.0)(0.36,100.0)(0.38,100.0)(0.4,100.0)(0.42,100.0)(0.44,100.0)(0.46,100.0)(0.48,100.0)(0.5,100.0)	
    };
    \addlegendentry{$k=4$}
\addplot[
    color=green,
    mark=diamond*,
		fill opacity=0.8,
  	draw opacity=0.8,
    ]
    coordinates {
    (0.0,1.25)(0.02,80.08275)(0.04,96.06787499999999)(0.06,99.2795)(0.08,99.741625)(0.1,99.98962499999999)(0.12,100.0)(0.14,100.0)(0.16,100.0)(0.18,100.0)(0.2,100.0)(0.22,100.0)(0.24,100.0)(0.26,100.0)(0.28,100.0)(0.3,100.0)(0.32,100.0)(0.34,100.0)(0.36,100.0)(0.38,100.0)(0.4,100.0)(0.42,100.0)(0.44,100.0)(0.46,100.0)(0.48,100.0)(0.5,100.0)
    };
    \addlegendentry{$k=n-1$}
\end{axis}
\end{tikzpicture}
\end{subfigure}\begin{subfigure}{.5\textwidth} 
	\centering
\begin{tikzpicture}[scale=0.85]
\begin{axis}[
    title={(f) $n = 100$},
    xlabel={},
    ylabel={},
    xmin=0, xmax=0.5,
    ymin=0, ymax=100,
    xtick= {0,0.1,0.2,0.3,0.4,0.5},
    ytick={0,20,40,60,80,100},
    legend pos=south east,
    ymajorgrids=true,
		xmajorgrids=true,
    grid style=dashed,
]

\addplot[
    color=black,
    mark size=3,
    mark=triangle*,
  	fill opacity=0.8,
  	draw opacity=0.8,
    ]
    coordinates {
    (0.0,1.0)(0.02,9.9392)(0.04,31.7178)(0.06,53.9256)(0.08,69.8127)(0.1,80.4828)(0.12,87.9219)(0.14,92.94149999999999)(0.16,96.221)(0.18,98.149)(0.2,99.1442)(0.22,99.65619999999998)(0.24,99.8802)(0.26,99.95740000000002)(0.28,99.9848)(0.3,99.99610000000001)(0.32,99.9984)(0.34,99.99960000000002)(0.36,100.0)(0.38,100.0)(0.4,100.0)(0.42,100.0)(0.44,100.0)(0.46,100.0)(0.48,100.0)(0.5,100.0)
    };
    \addlegendentry{$k=2$}

\addplot[
        color=red,
        mark=square*,
        fill opacity=0.8,
        draw opacity=0.8,
	    ]
	    coordinates {
    (0.0,1.0)(0.02,71.482)(0.04,95.8929)(0.06,99.536)(0.08,99.9526)(0.1,99.9897)(0.12,99.9999)(0.14,100.0)(0.16,100.0)(0.18,100.0)(0.2,100.0)(0.22,100.0)(0.24,100.0)(0.26,100.0)(0.28,100.0)(0.3,100.0)(0.32,100.0)(0.34,100.0)(0.36,100.0)(0.38,100.0)(0.4,100.0)(0.42,100.0)(0.44,100.0)(0.46,100.0)(0.48,100.0)(0.5,100.0)	
		};
	    \addlegendentry{$k=3$}

\addplot[
    color=blue,
    mark=otimes*,
    fill opacity=0.8,
  	draw opacity=0.8,
    ]
    coordinates {
    (0.0,1.0)(0.02,84.83139999999999)(0.04,98.1277)(0.06,99.6317)(0.08,99.96)(0.1,100.0)(0.12,100.0)(0.14,100.0)(0.16,100.0)(0.18,100.0)(0.2,100.0)(0.22,100.0)(0.24,100.0)(0.26,100.0)(0.28,100.0)(0.3,100.0)(0.32,100.0)(0.34,100.0)(0.36,100.0)(0.38,100.0)(0.4,100.0)(0.42,100.0)(0.44,100.0)(0.46,100.0)(0.48,100.0)(0.5,100.0)	
    };
    \addlegendentry{$k=4$}
\addplot[
    color=green,
    mark=diamond*,
		fill opacity=0.8,
  	draw opacity=0.8,
    ]
    coordinates {
    (0.0,1.0)(0.02,85.5846)(0.04,98.2273)(0.06,99.8193)(0.08,99.9602)(0.1,99.9999)(0.12,100.0)(0.14,100.0)(0.16,100.0)(0.18,100.0)(0.2,100.0)(0.22,100.0)(0.24,100.0)(0.26,100.0)(0.28,100.0)(0.3,100.0)(0.32,100.0)(0.34,100.0)(0.36,100.0)(0.38,100.0)(0.4,100.0)(0.42,100.0)(0.44,100.0)(0.46,100.0)(0.48,100.0)(0.5,100.0)
    };
    \addlegendentry{$k=n-1$}
\end{axis}
\end{tikzpicture}
\end{subfigure}
\caption{Percentage of alternatives that are $k$-kings in tournaments generated by the Condorcet random model with probability $p$, for $k\in\{2,3,4,n-1\}$ and different sizes $n$ of the tournament.
The horizontal and vertical axes correspond to the probability $p$ and the percentage, respectively. Averages are taken over 10000 generated tournaments.}
\label{fig:condorcet}
\end{figure}
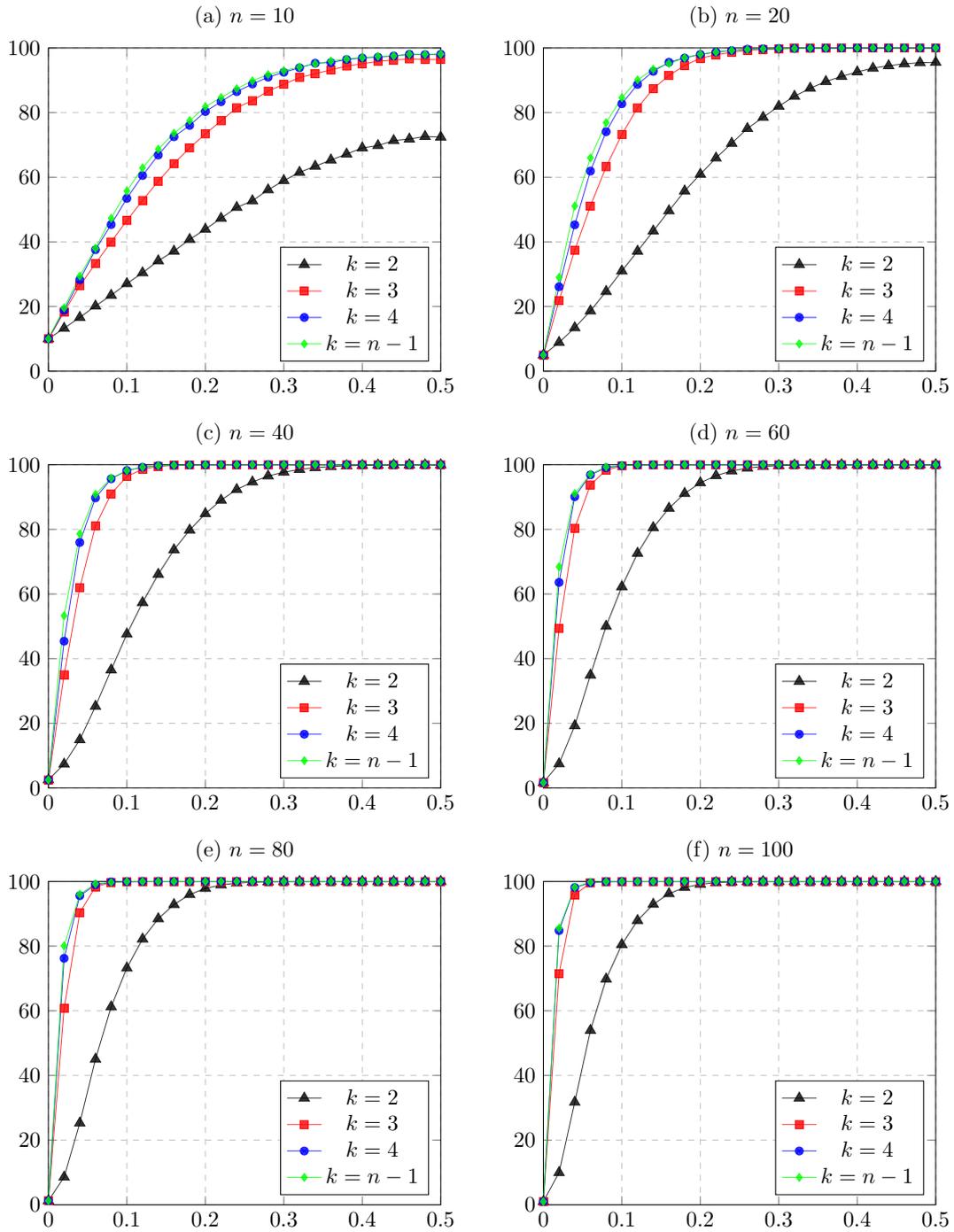

From the resulting graphs, it is clear that $3$-, $4$-, and $(n-1)$-kings behave much more similarly to one another than to $2$-kings, as our theoretical findings predict.
Note that $p=0$ always gives rise to a transitive tournament, which has only one $k$-king for every $k$, so the percentage of $k$-kings is always $(100/n)\%$ in this case.
As $p$ increases, the tournament becomes less skewed, and the number of $k$-kings also rises as a result.
This rise is significantly faster for larger $n$, since there are more intermediate alternatives through which an alternative can reach another alternative.
Indeed, for $n=100$, even when $p$ is only $0.1$, already $80\%$ of the alternatives are $2$-kings and $100\%$ are $3$, $4$, and $(n-1)$-kings.
A perhaps surprising observation is that $4$-kings behave less similarly to $3$-kings than to $(n-1)$-king, even though Table~\ref{table:summary} appears to predict the opposite outcome.
This behavior may be due to the facts that the difference between $\Theta(\log n/n)$ and $\Theta(\log\log n/n)$ is small, so the difference in the constant factors plays a more noticeable role.

Our results in Section~\ref{sec:number-kings} show that the transition from a small number to a large number of $2$-kings occurs in the range $p\in \Theta(\log n/n)$, with the threshold $p \ge 50\log n/n$ guaranteeing that at least $90\%$ of the alternatives are $2$-kings.
For $n=100$, this percentage of $2$-kings is reached when $p = 0.14\approx 3\log 100/100$, which means that there is likely to be room for improvement in Theorem~\ref{thm:2king-many} with respect to the constant factor.
A similar improvement is also plausible for Theorem~\ref{thm:2king-few}, as $0.1\log 100/100 \approx 0.005$, whereas our experiments exhibit that even when $p = 0.02$, only $10\%$ of the alternatives are $2$-kings.

For completeness, we also provide experimental results on the fraction of tournaments in which \emph{all} alternatives are $k$-kings in Appendix~\ref{app:additional}.

\subsection{Gap Model}

Next, we consider another random model for generating tournaments called the \emph{gap model} \cite{SaileSu20}.
Like in the Condorcet random model, in the gap model there is a parameter $p\in[0,1/2]$ and an ordering $x_1,\dots,x_n$ of the alternatives.
However, instead of the probability of $x_j$ dominating $x_i$ being $p$ for all $i < j$, this probability is now $0.5 - \frac{(0.5-p)(j-i)}{n-1}$.
As a result, although the domination probability remains $p$ when $(i,j) = (1,n)$, the probability is much closer to $0.5$ when $i$ and $j$ are close to each other.
In particular, even when $p=0$, the tournament can be far from transitive.
We conducted the same experiments for the gap model as we did for the Condorcet random model, with the exception being that we chose lower values of $n$, i.e., $n\in\{5,10,20,30,40,60\}$.
The reason behind this choice is that when $n > 60$, very close to $100\%$ of the alternatives are $k$-kings for every $k$.
The results of our experiments are shown in Figure~\ref{fig:gap}.
Note that for $n=5$, even though $4$-kings and $(n-1)$-kings coincide, their corresponding curves slightly differ because the tournaments are generated anew for each $k$.

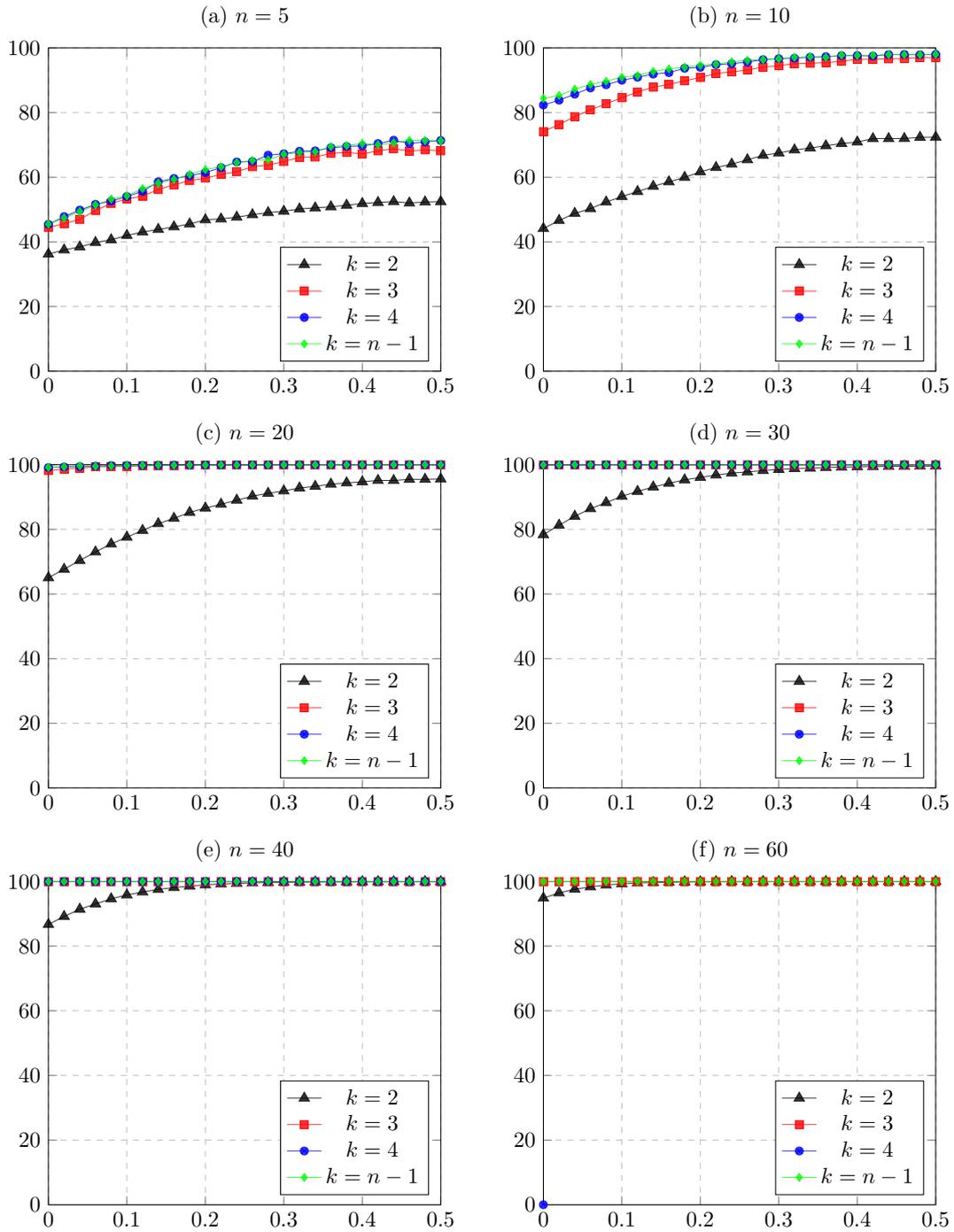
\begin{figure}
	\centering
	\begin{subfigure}{.5\textwidth} 
		\centering
\begin{tikzpicture}[scale=0.85]
\begin{axis}[
    title={(a) $n = 5$},
    xlabel={},
    ylabel={},
    xmin=0, xmax=0.5,
    ymin=0, ymax=100,
    xtick= {0,0.1,0.2,0.3,0.4,0.5},
    ytick={0,20,40,60,80,100},
    legend pos=south east,
    ymajorgrids=true,
		xmajorgrids=true,
    grid style=dashed,
]

\addplot[
    color=black,
	  mark size=3,
    mark=triangle*,
  	fill opacity=0.8,
  	draw opacity=0.8,
    ]
    coordinates {
    (0.0,36.226)(0.02,37.516)(0.04,38.438)(0.06,39.808)(0.08,40.662)(0.1,41.998)(0.12,43.056)(0.14,43.85)(0.16,44.658)(0.18,45.540000000000006)(0.2,46.833999999999996)(0.22,47.098)(0.24,47.653999999999996)(0.26,48.426)(0.28,49.05)(0.3,49.501999999999995)(0.32,50.182)(0.34,50.523999999999994)(0.36,50.86)(0.38,51.342)(0.4,51.874)(0.42,52.2)(0.44,52.474000000000004)(0.46,52.028)(0.48,52.338)(0.5,52.455999999999996)
    };
    \addlegendentry{$k=2$}

\addplot[
	    color=red,
	    mark=square*,
			fill opacity=0.8,
  	  draw opacity=0.8,
	    ]
	    coordinates {
	    (0.0,44.428000000000004)(0.02,45.67)(0.04,46.964)(0.06,49.744)(0.08,51.894000000000005)(0.1,53.312)(0.12,54.15)(0.14,56.198)(0.16,57.698)(0.18,59.004000000000005)(0.2,59.722)(0.22,60.891999999999996)(0.24,61.714)(0.26,63.251999999999995)(0.28,63.742)(0.3,64.868)(0.32,66.158)(0.34,66.23400000000001)(0.36,67.388)(0.38,67.684)(0.4,67.232)(0.42,68.16799999999999)(0.44,68.83200000000001)(0.46,68.09)(0.48,68.462)(0.5,68.15400000000001)
	    };
	    \addlegendentry{$k=3$}

\addplot[
    color=blue,
    mark=otimes*,
		fill opacity=0.8,
  	draw opacity=0.8,
    ]
    coordinates {
    (0.0,45.432)(0.02,47.832)(0.04,49.85000000000001)(0.06,51.61800000000001)(0.08,52.612)(0.1,54.04200000000001)(0.12,55.666)(0.14,58.604000000000006)(0.16,59.626)(0.18,60.45)(0.2,61.467999999999996)(0.22,62.977999999999994)(0.24,64.72200000000001)(0.26,64.87)(0.28,66.824)(0.3,67.262)(0.32,68.048)(0.34,68.16799999999999)(0.36,69.136)(0.38,69.526)(0.4,69.712)(0.42,70.418)(0.44,71.44800000000001)(0.46,70.468)(0.48,70.924)(0.5,71.372)
    };
    \addlegendentry{$k=4$}
\addplot[
    color=green,
    mark=diamond*,
		fill opacity=0.8,
  	draw opacity=0.8,
    ]
    coordinates {
    (0.0,45.498000000000005)(0.02,47.42)(0.04,49.382)(0.06,51.402)(0.08,53.246)(0.1,54.272000000000006)(0.12,56.492000000000004)(0.14,58.136)(0.16,59.302)(0.18,60.903999999999996)(0.2,62.366)(0.22,63.284000000000006)(0.24,64.512)(0.26,65.208)(0.28,65.44800000000001)(0.3,67.074)(0.32,67.684)(0.34,67.858)(0.36,69.486)(0.38,69.79400000000001)(0.4,70.45)(0.42,70.02799999999999)(0.44,70.66)(0.46,71.406)(0.48,71.4)(0.5,71.25800000000001)
    };
    \addlegendentry{$k=n-1$}
\end{axis}
\end{tikzpicture}
\end{subfigure}\begin{subfigure}{.5\textwidth} 
	\centering
\begin{tikzpicture}[scale=0.85]
\begin{axis}[
    title={(b) $n = 10$},
    xlabel={},
    ylabel={},
    xmin=0, xmax=0.5,
    ymin=0, ymax=100,
    xtick= {0,0.1,0.2,0.3,0.4,0.5},
    ytick={0,20,40,60,80,100},
    legend pos=south east,
    ymajorgrids=true,
		xmajorgrids=true,
    grid style=dashed,
]

\addplot[
    color=black,
    mark size=3,
    mark=triangle*,
  	fill opacity=0.8,
  	draw opacity=0.8,
    ]
    coordinates {
    (0.0,44.134)(0.02,46.626000000000005)(0.04,48.81400000000001)(0.06,50.29899999999999)(0.08,52.33200000000001)(0.1,54.028999999999996)(0.12,55.573)(0.14,57.194)(0.16,58.577)(0.18,59.959)(0.2,61.681)(0.22,63.01800000000001)(0.24,64.057)(0.26,65.413)(0.28,66.778)(0.3,67.50800000000001)(0.32,68.44500000000001)(0.34,69.044)(0.36,69.74199999999999)(0.38,70.439)(0.4,70.894)(0.42,71.953)(0.44,71.95700000000001)(0.46,71.98299999999999)(0.48,72.35)(0.5,72.393)
    };
    \addlegendentry{$k=2$}

\addplot[
	    color=red,
        mark=square*,
        fill opacity=0.8,
        draw opacity=0.8,
	    ]
	    coordinates {
	    (0.0,74.138)(0.02,76.328)(0.04,78.747)(0.06,80.854)(0.08,82.729)(0.1,84.60100000000001)(0.12,86.372)(0.14,87.854)(0.16,88.703)(0.18,89.811)(0.2,90.83000000000001)(0.22,92.093)(0.24,92.578)(0.26,93.164)(0.28,94.04400000000001)(0.3,94.526)(0.32,95.11100000000002)(0.34,95.24699999999999)(0.36,95.44800000000001)(0.38,95.854)(0.4,96.467)(0.42,96.478)(0.44,96.614)(0.46,96.70700000000001)(0.48,97.03799999999998)(0.5,96.953)
    	   };
	    \addlegendentry{$k=3$}

\addplot[
    color=blue,
    mark=otimes*,
    fill opacity=0.8,
  	draw opacity=0.8,
    ]
    coordinates {
     (0.0,82.333)(0.02,83.841)(0.04,85.696)(0.06,87.619)(0.08,88.59700000000001)(0.1,89.999)(0.12,90.91900000000001)(0.14,91.90200000000002)(0.16,92.365)(0.18,93.71000000000001)(0.2,94.002)(0.22,94.83599999999998)(0.24,95.105)(0.26,95.624)(0.28,96.37100000000001)(0.3,96.661)(0.32,96.821)(0.34,97.121)(0.36,97.29499999999999)(0.38,97.598)(0.4,97.626)(0.42,97.479)(0.44,97.949)(0.46,97.99699999999999)(0.48,97.88500000000002)(0.5,98.018)
    };
    \addlegendentry{$k=4$}
\addplot[
    color=green,
    mark=diamond*,
		fill opacity=0.8,
  	draw opacity=0.8,
    ]
    coordinates {
    (0.0,84.397)(0.02,85.30099999999999)(0.04,87.27399999999999)(0.06,88.708)(0.08,89.691)(0.1,90.96000000000001)(0.12,91.449)(0.14,92.72200000000001)(0.16,93.49)(0.18,94.06800000000001)(0.2,94.651)(0.22,95.001)(0.24,95.662)(0.26,96.184)(0.28,96.37700000000001)(0.3,96.63399999999999)(0.32,97.10900000000001)(0.34,97.268)(0.36,97.08500000000001)(0.38,97.777)(0.4,97.65700000000001)(0.42,97.718)(0.44,97.941)(0.46,97.90200000000002)(0.48,97.886)(0.5,98.006)
    };
    \addlegendentry{$k=n-1$}
\end{axis}
\end{tikzpicture}
\end{subfigure}

\vspace{2mm}

	\centering
	\begin{subfigure}{.5\textwidth} 
		\centering
\begin{tikzpicture}[scale=0.85]
\begin{axis}[
    title={(c) $n = 20$},
    xlabel={},
    ylabel={},
    xmin=0, xmax=0.5,
    ymin=0, ymax=100,
    xtick= {0,0.1,0.2,0.3,0.4,0.5},
    ytick={0,20,40,60,80,100},
    legend pos=south east,
    ymajorgrids=true,
		xmajorgrids=true,
    grid style=dashed,
]

\addplot[
    color=black,
	  mark size=3,
    mark=triangle*,
  	fill opacity=0.8,
  	draw opacity=0.8,
    ]
    coordinates {
    (0.0,65.0395)(0.02,67.6645)(0.04,70.3625)(0.06,73.0095)(0.08,75.4795)(0.1,77.59450000000001)(0.12,79.67699999999999)(0.14,81.76650000000001)(0.16,83.4195)(0.18,85.2535)(0.2,86.587)(0.22,87.7775)(0.24,89.03299999999999)(0.26,90.2755)(0.28,91.16300000000001)(0.3,91.89450000000001)(0.32,92.8065)(0.34,93.318)(0.36,93.988)(0.38,94.373)(0.4,94.712)(0.42,95.15700000000001)(0.44,95.1365)(0.46,95.47)(0.48,95.4865)(0.5,95.59900000000002)
    };
    \addlegendentry{$k=2$}

\addplot[
	    color=red,
        mark=square*,
        fill opacity=0.8,
        draw opacity=0.8,
	    ]
	    coordinates {
    (0.0,98.183)(0.02,98.69300000000001)(0.04,98.9225)(0.06,99.351)(0.08,99.422)(0.1,99.50649999999999)(0.12,99.67599999999999)(0.14,99.761)(0.16,99.76950000000001)(0.18,99.85849999999999)(0.2,99.878)(0.22,99.926)(0.24,99.91099999999999)(0.26,99.9505)(0.28,99.954)(0.3,99.9765)(0.32,99.96900000000001)(0.34,99.987)(0.36,99.9995)(0.38,99.9995)(0.4,99.989)(0.42,99.98849999999999)(0.44,100.0)(0.46,99.99000000000001)(0.48,99.9895)(0.5,99.98849999999999)
	    };
	    \addlegendentry{$k=3$}

\addplot[
    color=blue,
    mark=otimes*,
    fill opacity=0.8,
  	draw opacity=0.8,
    ]
    coordinates {
    (0.0,99.123)(0.02,99.319)(0.04,99.4275)(0.06,99.549)(0.08,99.7495)(0.1,99.77650000000001)(0.12,99.79150000000001)(0.14,99.82950000000001)(0.16,99.87899999999999)(0.18,99.88950000000001)(0.2,99.92150000000001)(0.22,99.9025)(0.24,99.9675)(0.26,99.94949999999999)(0.28,99.95150000000001)(0.3,99.9795)(0.32,99.99000000000001)(0.34,99.97149999999999)(0.36,99.99050000000001)(0.38,99.9985)(0.4,99.98049999999999)(0.42,99.98100000000001)(0.44,100.0)(0.46,100.0)(0.48,100.0)(0.5,99.9895)
    };
    \addlegendentry{$k=4$}
\addplot[
    color=green,
    mark=diamond*,
		fill opacity=0.8,
  	draw opacity=0.8,
    ]
    coordinates {
    (0.0,99.20450000000001)(0.02,99.36299999999999)(0.04,99.681)(0.06,99.511)(0.08,99.636)(0.1,99.77000000000001)(0.12,99.8185)(0.14,99.88950000000001)(0.16,99.84049999999999)(0.18,99.93100000000001)(0.2,99.922)(0.22,99.958)(0.24,99.93)(0.26,99.9605)(0.28,99.96950000000001)(0.3,99.97049999999999)(0.32,99.999)(0.34,99.97149999999999)(0.36,99.99050000000001)(0.38,99.99000000000001)(0.4,100.0)(0.42,100.0)(0.44,100.0)(0.46,99.9985)(0.48,99.99050000000001)(0.5,99.9895)
    };
    \addlegendentry{$k=n-1$}
\end{axis}
\end{tikzpicture}
\end{subfigure}\begin{subfigure}{.5\textwidth} 
	\centering
\begin{tikzpicture}[scale=0.85]
\begin{axis}[
    title={(d) $n = 30$},
    xlabel={},
    ylabel={},
    xmin=0, xmax=0.5,
    ymin=0, ymax=100,
    xtick= {0,0.1,0.2,0.3,0.4,0.5},
    ytick={0,20,40,60,80,100},
    legend pos=south east,
    ymajorgrids=true,
	  xmajorgrids=true,
    grid style=dashed,
]

\addplot[
    color=black,
    mark size=3,
    mark=triangle*,
  	fill opacity=0.8,
  	draw opacity=0.8,
    ]
    coordinates {
    (0.0,78.36733333333333)(0.02,81.28733333333334)(0.04,84.05466666666666)(0.06,86.40566666666668)(0.08,88.27433333333333)(0.1,90.23966666666666)(0.12,91.75133333333333)(0.14,93.10733333333334)(0.16,94.26766666666667)(0.18,95.31433333333332)(0.2,96.09533333333334)(0.22,96.81633333333333)(0.24,97.40166666666667)(0.26,97.784)(0.28,98.19633333333333)(0.3,98.54166666666667)(0.32,98.80933333333334)(0.34,98.93766666666667)(0.36,99.14766666666667)(0.38,99.25499999999998)(0.4,99.36999999999999)(0.42,99.42533333333334)(0.44,99.49066666666667)(0.46,99.50933333333333)(0.48,99.566)(0.5,99.587)
    };
    \addlegendentry{$k=2$}

\addplot[
	    color=red,
        mark=square*,
        fill opacity=0.8,
        draw opacity=0.8,
	    ]
	    coordinates {
		(0.0,99.92699999999999)(0.02,99.943)(0.04,99.974)(0.06,99.96566666666666)(0.08,99.97599999999998)(0.1,99.99666666666667)(0.12,99.99)(0.14,99.989)(0.16,100.0)(0.18,99.99933333333334)(0.2,100.0)(0.22,99.99966666666667)(0.24,100.0)(0.26,100.0)(0.28,99.99033333333334)(0.3,100.0)(0.32,100.0)(0.34,100.0)(0.36,100.0)(0.38,100.0)(0.4,100.0)(0.42,100.0)(0.44,100.0)(0.46,100.0)(0.48,100.0)(0.5,100.0)
		};
	    \addlegendentry{$k=3$}

\addplot[
    color=blue,
    mark=otimes*,
    fill opacity=0.8,
  	draw opacity=0.8,
    ]
    coordinates {
    (0.0,99.96000000000001)(0.02,99.95066666666666)(0.04,99.96000000000001)(0.06,99.97966666666666)(0.08,99.98066666666666)(0.1,99.99033333333334)(0.12,99.98066666666666)(0.14,99.99033333333334)(0.16,99.99966666666667)(0.18,100.0)(0.2,100.0)(0.22,100.0)(0.24,100.0)(0.26,100.0)(0.28,100.0)(0.3,100.0)(0.32,100.0)(0.34,100.0)(0.36,100.0)(0.38,100.0)(0.4,100.0)(0.42,100.0)(0.44,100.0)(0.46,100.0)(0.48,100.0)(0.5,100.0)
    };
    \addlegendentry{$k=4$}
\addplot[
    color=green,
    mark=diamond*,
		fill opacity=0.8,
  	draw opacity=0.8,
    ]
    coordinates {
    (0.0,99.96000000000001)(0.02,99.95166666666665)(0.04,99.98066666666666)(0.06,100.0)(0.08,99.98966666666666)(0.1,99.99)(0.12,99.98066666666666)(0.14,100.0)(0.16,99.99966666666667)(0.18,100.0)(0.2,100.0)(0.22,100.0)(0.24,100.0)(0.26,100.0)(0.28,100.0)(0.3,100.0)(0.32,100.0)(0.34,100.0)(0.36,100.0)(0.38,100.0)(0.4,100.0)(0.42,100.0)(0.44,100.0)(0.46,100.0)(0.48,100.0)(0.5,100.0)
    };
    \addlegendentry{$k=n-1$}
\end{axis}
\end{tikzpicture}
\end{subfigure}

\vspace{2mm}

	\begin{subfigure}{.5\textwidth} 
		\centering
\begin{tikzpicture}[scale=0.85]
\begin{axis}[
    title={(e) $n = 40$},
    xlabel={},
    ylabel={},
    xmin=0, xmax=0.5,
    ymin=0, ymax=100,
    xtick= {0,0.1,0.2,0.3,0.4,0.5},
    ytick={0,20,40,60,80,100},
    legend pos=south east,
    ymajorgrids=true,
		xmajorgrids=true,
    grid style=dashed,
]

\addplot[
    color=black,
    mark size=3,
    mark=triangle*,
  	fill opacity=0.8,
  	draw opacity=0.8,
    ]
    coordinates {
    (0.0,86.72025)(0.02,89.17350000000002)(0.04,91.423)(0.06,93.05325)(0.08,94.6505)(0.1,95.841)(0.12,96.74699999999999)(0.14,97.55775)(0.16,98.1165)(0.18,98.58399999999999)(0.2,98.94800000000001)(0.22,99.23325)(0.24,99.42399999999999)(0.26,99.58475)(0.28,99.68575)(0.3,99.76875)(0.32,99.83174999999999)(0.34,99.86475)(0.36,99.89375000000001)(0.38,99.91625000000002)(0.4,99.93775)(0.42,99.95075)(0.44,99.955)(0.46,99.962)(0.48,99.96525)(0.5,99.96600000000001)
    };
    \addlegendentry{$k=2$}

\addplot[
        color=red,
        mark=square*,
        fill opacity=0.8,
        draw opacity=0.8,
	    ]
	    coordinates {
    (0.0,99.99924999999999)(0.02,99.99875)(0.04,100.0)(0.06,100.0)(0.08,100.0)(0.1,100.0)(0.12,100.0)(0.14,100.0)(0.16,100.0)(0.18,100.0)(0.2,100.0)(0.22,100.0)(0.24,100.0)(0.26,100.0)(0.28,100.0)(0.3,100.0)(0.32,100.0)(0.34,100.0)(0.36,100.0)(0.38,100.0)(0.4,100.0)(0.42,100.0)(0.44,100.0)(0.46,100.0)(0.48,100.0)(0.5,100.0)	
		};
	    \addlegendentry{$k=3$}

\addplot[
    color=blue,
    mark=otimes*,
    fill opacity=0.8,
  	draw opacity=0.8,
    ]
    coordinates {
    (0.0,100.0)(0.02,100.0)(0.04,100.0)(0.06,100.0)(0.08,100.0)(0.1,100.0)(0.12,100.0)(0.14,100.0)(0.16,100.0)(0.18,100.0)(0.2,100.0)(0.22,100.0)(0.24,100.0)(0.26,100.0)(0.28,100.0)(0.3,100.0)(0.32,100.0)(0.34,100.0)(0.36,100.0)(0.38,100.0)(0.4,100.0)(0.42,100.0)(0.44,100.0)(0.46,100.0)(0.48,100.0)(0.5,100.0)
    };
    \addlegendentry{$k=4$}
\addplot[
    color=green,
    mark=diamond*,
		fill opacity=0.8,
  	draw opacity=0.8,
    ]
    coordinates {
    (0.0,100.0)(0.02,100.0)(0.04,100.0)(0.06,100.0)(0.08,100.0)(0.1,100.0)(0.12,100.0)(0.14,100.0)(0.16,100.0)(0.18,100.0)(0.2,100.0)(0.22,100.0)(0.24,100.0)(0.26,100.0)(0.28,100.0)(0.3,100.0)(0.32,100.0)(0.34,100.0)(0.36,100.0)(0.38,100.0)(0.4,100.0)(0.42,100.0)(0.44,100.0)(0.46,100.0)(0.48,100.0)(0.5,100.0)
    };
    \addlegendentry{$k=n-1$}
\end{axis}
\end{tikzpicture}
\end{subfigure}\begin{subfigure}{.5\textwidth} 
	\centering
\begin{tikzpicture}[scale=0.85]
\begin{axis}[
    title={(f) $n = 60$},
    xlabel={},
    ylabel={},
    xmin=0, xmax=0.5,
    ymin=0, ymax=100,
    xtick= {0,0.1,0.2,0.3,0.4,0.5},
    ytick={0,20,40,60,80,100},
    legend pos=south east,
    ymajorgrids=true,
		xmajorgrids=true,
    grid style=dashed,
]

\addplot[
    color=black,
    mark size=3,
    mark=triangle*,
  	fill opacity=0.8,
  	draw opacity=0.8,
    ]
    coordinates {
    (0.0,94.86783333333332)(0.02,96.4665)(0.04,97.59750000000001)(0.06,98.34733333333334)(0.08,98.91066666666667)(0.1,99.28483333333334)(0.12,99.54466666666666)(0.14,99.69900000000001)(0.16,99.80016666666667)(0.18,99.8835)(0.2,99.92783333333334)(0.22,99.95666666666666)(0.24,99.97366666666667)(0.26,99.98366666666668)(0.28,99.99216666666666)(0.3,99.99316666666667)(0.32,99.99499999999999)(0.34,99.99733333333332)(0.36,99.99900000000001)(0.38,99.99833333333335)(0.4,99.99916666666667)(0.42,99.99949999999998)(0.44,99.99983333333333)(0.46,99.99983333333333)(0.48,100.0)(0.5,99.99983333333333)
    };
    \addlegendentry{$k=2$}

\addplot[
        color=red,
        mark=square*,
        fill opacity=0.8,
        draw opacity=0.8,
	    ]
	    coordinates {
    (0.0,100.0)(0.02,100.0)(0.04,100.0)(0.06,100.0)(0.08,100.0)(0.1,100.0)(0.12,100.0)(0.14,100.0)(0.16,100.0)(0.18,100.0)(0.2,100.0)(0.22,100.0)(0.24,100.0)(0.26,100.0)(0.28,100.0)(0.3,100.0)(0.32,100.0)(0.34,100.0)(0.36,100.0)(0.38,100.0)(0.4,100.0)(0.42,100.0)(0.44,100.0)(0.46,100.0)(0.48,100.0)(0.5,100.0)	
		};
	    \addlegendentry{$k=3$}

\addplot[
    color=blue,
    mark=otimes*,
    fill opacity=0.8,
  	draw opacity=0.8,
    ]
    coordinates {
    (0,0)	
    };
    \addlegendentry{$k=4$}
\addplot[
    color=green,
    mark=diamond*,
		fill opacity=0.8,
  	draw opacity=0.8,
    ]
    coordinates {
    (0.0,100.0)(0.02,100.0)(0.04,100.0)(0.06,100.0)(0.08,100.0)(0.1,100.0)(0.12,100.0)(0.14,100.0)(0.16,100.0)(0.18,100.0)(0.2,100.0)(0.22,100.0)(0.24,100.0)(0.26,100.0)(0.28,100.0)(0.3,100.0)(0.32,100.0)(0.34,100.0)(0.36,100.0)(0.38,100.0)(0.4,100.0)(0.42,100.0)(0.44,100.0)(0.46,100.0)(0.48,100.0)(0.5,100.0)
    };
    \addlegendentry{$k=n-1$}
\end{axis}
\end{tikzpicture}
\end{subfigure}
\caption{Percentage of alternatives that are $k$-kings in tournaments generated by the gap model with probability $p$, for $k\in\{2,3,4,n-1\}$ and different sizes $n$ of the tournament.
The horizontal and vertical axes correspond to the probability $p$ and the percentage, respectively. Averages are taken over 10000 generated tournaments.}
\label{fig:gap}
\end{figure}

Similarly to the Condorcet random model, our results for the gap model reveal a clear separation between $2$-kings and $k$-kings for $k\ge 3$.
Moreover, the average number of $k$-kings is much higher in the gap model (especially for low values of $p$), which indicates that the proportion of $k$-kings in real-life tournaments is likely to be higher than in tournaments generated according to the Condorcet random model.
We remark that even when $p = 0$, the experimental findings suggest that the percentage of $k$-kings converges to $100\%$ for every $k$ as $n$ grows.
This result has \emph{not} been established theoretically by our work (or any prior work), because the probability that alternative $x_n$ dominates alternative $x_1$ is $0$ when $p = 0$, so the results on the generalized random model do not apply.
Establishing this result and studying the gap model theoretically is an interesting direction for future work.

As with the Condorcet random model, we present results on the fraction of tournaments in which all alternatives are $k$-kings in Appendix~\ref{app:additional}.

\section{Concluding Remarks}

In this paper, we have extensively investigated the behavior of generalized kings and single-elimination winners in random tournaments in view of their discriminative power.
Our results reveal surprisingly clear distinctions between $2$-kings and $k$-kings for $k\geq 3$, and illustrate why manipulating a single-elimination tournament is often possible in practice despite the problem being NP-hard.
All of the bounds that we obtained are asymptotically tight except for the bound for $5$-kings in the generalized random model (Theorem~\ref{thm:5king-positive-generalized}); one could try to close this gap.

An exciting future direction in our view is to study $k$-kings and single-elimination winners with respect to axiomatic and computational properties, as is commonly done for other tournament solutions \cite{BrandtBrHa16,Laslier97}.
For example, the set of $2$-kings (i.e., the uncovered set) is known to be the finest tournament solution satisfying Condorcet consistency, neutrality, and expansion \cite{Moulin86}.
Which axioms does the set of $3$-kings satisfy, and can we characterize it by a collection of axioms?
One could also study the relationship between these tournament solutions and traditional ones---this was partially done by Kim et al.~\cite{KimSuVa17}, who showed for instance that any alternative in the Copeland set or the Slater set can always win a single-elimination tournament.
Another possible avenue is to extend our results to other stochastic models for tournaments---several interesting models have been studied experimentally by Brandt and Seedig~\cite{BrandtSe16} and Brill et al.~\cite{BrillScSu22}.
Such questions illustrate the richness of tournaments and probabilistic models, which we expect to give rise to further fruitful research.

\section*{Acknowledgments}

This work was partially supported by an NUS Start-up Grant.
We would like to thank the anonymous reviewers of IJCAI 2021 and JAAMAS for their valuable comments.

\bibliographystyle{abbrv}
\bibliography{main}

\appendix

\section{Additional Experiments}
\label{app:additional}

In this appendix, we present experimental results on the fraction of tournaments in which \emph{all} alternatives are $k$-kings.
This complements the results in Section~\ref{sec:experiments}, which show the average fraction of $k$-kings across tournaments.
The setup of our experiments is identical to that in Section~\ref{sec:experiments}---we used the same sets of parameters $n$, $p$, and $k$ for each of the two random models, and for each of these sets, we generated $10000$ tournaments.
The plots are shown in Figures~\ref{fig:condorcet-choose-all} and \ref{fig:gap-choose-all}.

Once again, as our theoretical results predict, $2$-kings are significantly more selective than $k$-kings for $k\ge 3$.
In fact, the distinction is even clearer here than in Section~\ref{sec:experiments}.
Indeed, for $n=20$ and $p = 0.5$ in the gap model, all alternatives are $2$-kings in only $46\%$ of the tournaments, while even when $p = 0$, this is already the case for over $84\%$ of the tournaments with respect to $3$-kings and over $98\%$ with respect to $4$-kings (Figure~\ref{fig:gap-choose-all}(c)).
It is also worth noting that the values of $p$ for which all alternatives become $k$-kings in the Condorcet random model are quite close to the analytical predictions: when $n = 100$, we have $\sqrt{\log n/n} \approx 0.215$ and $\log n/n \approx 0.046$, which closely match the transitional probabilities in Figure~\ref{fig:condorcet-choose-all}(f) for $k = 2$ and $k\in\{3,4\}$, respectively.

The final remark that we would like to make is that the average fraction of $k$-kings and the fraction of tournaments that consist exclusively of $k$-kings can sometimes differ by a substantial margin.
For instance, when $n = 100$ and $p = 0.16$ in the Condorcet random model, even though $96\%$ of the alternatives are $2$-kings on average (Figure~\ref{fig:condorcet}(f)), all alternatives are $2$-kings in only $4\%$ of the tournaments (Figure~\ref{fig:condorcet-choose-all}(f)).
This means that in a large majority of tournaments generated according to these parameters, the percentage of $2$-kings is very close to, but strictly less than, $100\%$.

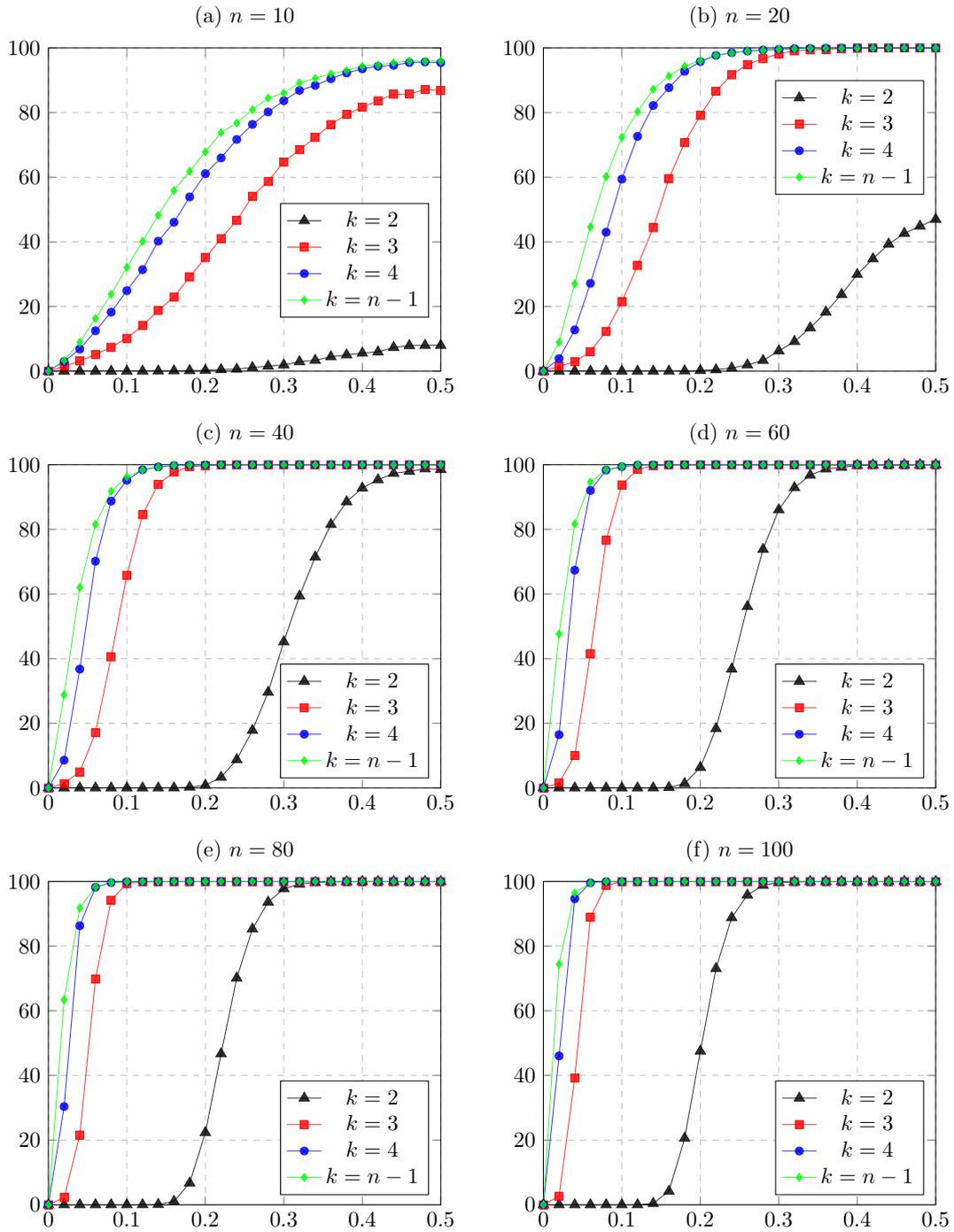
\begin{figure}
	\centering
	\begin{subfigure}{.5\textwidth} 
		\centering
\begin{tikzpicture}[scale=0.85]
\begin{axis}[
    title={(a) $n = 10$},
    xlabel={},
    ylabel={},
    xmin=0, xmax=0.5,
    ymin=0, ymax=100,
    xtick= {0,0.1,0.2,0.3,0.4,0.5},
    ytick={0,20,40,60,80,100},
    legend style={at={(0.97,0.52)}},
    ymajorgrids=true,
		xmajorgrids=true,
    grid style=dashed,
]

\addplot[
    color=black,
	  mark size=3,
    mark=triangle*,
  	fill opacity=0.8,
  	draw opacity=0.8,
    ]
    coordinates {
    (0.0,0.0)(0.02,0.0)(0.04,0.0)(0.06,0.0)(0.08,0.0)(0.1,0.0)(0.12,0.02)(0.14,0.05)(0.16,0.1)(0.18,0.16)(0.2,0.27)(0.22,0.5)(0.24,0.62)(0.26,1.1199999999999999)(0.28,1.55)(0.3,1.9300000000000002)(0.32,2.93)(0.34,3.36)(0.36,4.45)(0.38,5.029999999999999)(0.4,5.58)(0.42,5.94)(0.44,7.340000000000001)(0.46,7.91)(0.48,7.969999999999999)(0.5,7.9799999999999995)
    };
    \addlegendentry{$k=2$}

\addplot[
	    color=red,
	    mark=square*,
			fill opacity=0.8,
  	  draw opacity=0.8,
	    ]
	    coordinates {
	    (0.0,0.0)(0.02,1.47)(0.04,3.1399999999999997)(0.06,5.13)(0.08,7.430000000000001)(0.1,10.15)(0.12,14.08)(0.14,18.790000000000003)(0.16,22.95)(0.18,29.099999999999998)(0.2,35.099999999999994)(0.22,41.02)(0.24,46.63)(0.26,54.11)(0.28,58.709999999999994)(0.3,64.69)(0.32,68.51)(0.34,72.31)(0.36,76.33)(0.38,79.43)(0.4,81.67999999999999)(0.42,83.65)(0.44,85.67)(0.46,85.72)(0.48,87.13)(0.5,86.79)
	    };
	    \addlegendentry{$k=3$}

\addplot[
    color=blue,
    mark=otimes*,
		fill opacity=0.8,
  	draw opacity=0.8,
    ]
    coordinates {
    (0.0,0.0)(0.02,3.02)(0.04,6.819999999999999)(0.06,12.5)(0.08,18.25)(0.1,24.92)(0.12,31.4)(0.14,40.22)(0.16,46.07)(0.18,53.910000000000004)(0.2,61.12)(0.22,65.97)(0.24,71.69)(0.26,76.34)(0.28,80.2)(0.3,83.66)(0.32,86.86)(0.34,88.4)(0.36,90.46)(0.38,92.24)(0.4,93.52000000000001)(0.42,94.33)(0.44,94.61)(0.46,95.47)(0.48,95.65)(0.5,95.44)
    };
    \addlegendentry{$k=4$}
    
\addplot[
    color=green,
    mark=diamond*,
		fill opacity=0.8,
  	draw opacity=0.8,
    ]
    coordinates {
    (0.0,0.0)(0.02,3.29)(0.04,8.93)(0.06,16.33)(0.08,23.82)(0.1,32.09)(0.12,40.11)(0.14,48.24)(0.16,55.900000000000006)(0.18,61.82)(0.2,67.85)(0.22,73.82)(0.24,76.74)(0.26,80.89)(0.28,84.48)(0.3,86.00999999999999)(0.32,89.16)(0.34,90.56)(0.36,91.97999999999999)(0.38,92.99)(0.4,94.36)(0.42,94.74000000000001)(0.44,95.41)(0.46,95.93)(0.48,95.88)(0.5,95.92)
    };
    \addlegendentry{$k=n-1$}
\end{axis}
\end{tikzpicture}
\end{subfigure}\begin{subfigure}{.5\textwidth} 
	\centering
\begin{tikzpicture}[scale=0.85]
\begin{axis}[
    title={(b) $n = 20$},
    xlabel={},
    ylabel={},
    xmin=0, xmax=0.5,
    ymin=0, ymax=100,
    xtick= {0,0.1,0.2,0.3,0.4,0.5},
    ytick={0,20,40,60,80,100},
    legend style={at={(0.97,0.9)}},
    ymajorgrids=true,
		xmajorgrids=true,
    grid style=dashed,
]

\addplot[
    color=black,
    mark size=3,
    mark=triangle*,
  	fill opacity=0.8,
  	draw opacity=0.8,
    ]
    coordinates {
    (0.0,0.0)(0.02,0.0)(0.04,0.0)(0.06,0.0)(0.08,0.0)(0.1,0.0)(0.12,0.0)(0.14,0.01)(0.16,0.02)(0.18,0.06999999999999999)(0.2,0.22)(0.22,0.44)(0.24,0.96)(0.26,1.94)(0.28,3.3300000000000005)(0.3,6.279999999999999)(0.32,9.16)(0.34,13.41)(0.36,18.29)(0.38,23.74)(0.4,29.93)(0.42,34.77)(0.44,39.32)(0.46,42.65)(0.48,44.87)(0.5,46.989999999999995)
    };
    \addlegendentry{$k=2$}

\addplot[
	    color=red,
        mark=square*,
        fill opacity=0.8,
        draw opacity=0.8,
	    ]
	    coordinates {
	    (0.0,0.0)(0.02,1.5)(0.04,2.9000000000000004)(0.06,6.08)(0.08,12.33)(0.1,21.42)(0.12,32.64)(0.14,44.37)(0.16,59.63)(0.18,70.72)(0.2,79.13)(0.22,86.63)(0.24,91.64999999999999)(0.26,94.87)(0.28,96.67999999999999)(0.3,98.2)(0.32,99.06)(0.34,99.39)(0.36,99.53999999999999)(0.38,99.8)(0.4,99.86)(0.42,99.96000000000001)(0.44,99.95)(0.46,100.0)(0.48,99.98)(0.5,99.98)
    	   };
	    \addlegendentry{$k=3$}

\addplot[
    color=blue,
    mark=otimes*,
    fill opacity=0.8,
  	draw opacity=0.8,
    ]
    coordinates {
     (0.0,0.0)(0.02,3.8899999999999997)(0.04,12.770000000000001)(0.06,27.169999999999998)(0.08,42.99)(0.1,59.41)(0.12,72.65)(0.14,82.19)(0.16,87.7)(0.18,92.75)(0.2,95.77)(0.22,97.68)(0.24,98.52)(0.26,99.00999999999999)(0.28,99.36)(0.3,99.56)(0.32,99.81)(0.34,99.86)(0.36,99.92)(0.38,99.92999999999999)(0.4,99.99)(0.42,99.97)(0.44,99.98)(0.46,100.0)(0.48,99.99)(0.5,99.98)
    };
    \addlegendentry{$k=4$}
    
\addplot[
    color=green,
    mark=diamond*,
		fill opacity=0.8,
  	draw opacity=0.8,
    ]
    coordinates {
    (0.0,0.0)(0.02,8.94)(0.04,27.05)(0.06,44.61)(0.08,60.199999999999996)(0.1,72.36)(0.12,80.27)(0.14,87.22)(0.16,91.32000000000001)(0.18,94.37)(0.2,96.07)(0.22,97.72)(0.24,98.4)(0.26,99.02)(0.28,99.26)(0.3,99.53)(0.32,99.72)(0.34,99.86)(0.36,99.89)(0.38,99.89)(0.4,99.94)(0.42,99.96000000000001)(0.44,99.98)(0.46,99.99)(0.48,99.98)(0.5,99.98)
    };
    \addlegendentry{$k=n-1$}
\end{axis}
\end{tikzpicture}
\end{subfigure}

\vspace{2mm}

	\centering
	\begin{subfigure}{.5\textwidth} 
		\centering
\begin{tikzpicture}[scale=0.85]
\begin{axis}[
    title={(c) $n = 40$},
    xlabel={},
    ylabel={},
    xmin=0, xmax=0.5,
    ymin=0, ymax=100,
    xtick= {0,0.1,0.2,0.3,0.4,0.5},
    ytick={0,20,40,60,80,100},
    legend pos=south east,
    ymajorgrids=true,
		xmajorgrids=true,
    grid style=dashed,
]

\addplot[
    color=black,
	  mark size=3,
    mark=triangle*,
  	fill opacity=0.8,
  	draw opacity=0.8,
    ]
    coordinates {
    (0.0,0.0)(0.02,0.0)(0.04,0.0)(0.06,0.0)(0.08,0.0)(0.1,0.0)(0.12,0.0)(0.14,0.0)(0.16,0.02)(0.18,0.26)(0.2,0.83)(0.22,3.3000000000000003)(0.24,8.75)(0.26,17.810000000000002)(0.28,29.609999999999996)(0.3,45.22)(0.32,59.37)(0.34,71.44)(0.36,81.52000000000001)(0.38,88.53)(0.4,92.78999999999999)(0.42,95.35)(0.44,97.33000000000001)(0.46,97.97)(0.48,98.72)(0.5,98.5)
    };
    \addlegendentry{$k=2$}

\addplot[
	    color=red,
        mark=square*,
        fill opacity=0.8,
        draw opacity=0.8,
	    ]
	    coordinates {
    (0.0,0.0)(0.02,1.38)(0.04,4.96)(0.06,17.080000000000002)(0.08,40.5)(0.1,65.77)(0.12,84.53)(0.14,93.94)(0.16,97.69)(0.18,99.44)(0.2,99.7)(0.22,99.95)(0.24,99.99)(0.26,100.0)(0.28,100.0)(0.3,100.0)(0.32,100.0)(0.34,100.0)(0.36,100.0)(0.38,100.0)(0.4,100.0)(0.42,100.0)(0.44,100.0)(0.46,100.0)(0.48,100.0)(0.5,100.0)
	    };
	    \addlegendentry{$k=3$}

\addplot[
    color=blue,
    mark=otimes*,
    fill opacity=0.8,
  	draw opacity=0.8,
    ]
    coordinates {
    (0.0,0.0)(0.02,8.58)(0.04,36.76)(0.06,70.15)(0.08,88.75)(0.1,95.21)(0.12,98.44000000000001)(0.14,99.33)(0.16,99.77000000000001)(0.18,99.92)(0.2,99.97)(0.22,99.99)(0.24,100.0)(0.26,99.99)(0.28,100.0)(0.3,100.0)(0.32,100.0)(0.34,100.0)(0.36,100.0)(0.38,100.0)(0.4,100.0)(0.42,100.0)(0.44,100.0)(0.46,100.0)(0.48,100.0)(0.5,100.0)
    };
    \addlegendentry{$k=4$}
\addplot[
    color=green,
    mark=diamond*,
		fill opacity=0.8,
  	draw opacity=0.8,
    ]
    coordinates {
   (0.0,0.0)(0.02,28.82)(0.04,61.980000000000004)(0.06,81.54)(0.08,91.77)(0.1,96.26)(0.12,98.38)(0.14,99.24)(0.16,99.78)(0.18,99.9)(0.2,99.92999999999999)(0.22,100.0)(0.24,100.0)(0.26,100.0)(0.28,100.0)(0.3,100.0)(0.32,100.0)(0.34,100.0)(0.36,100.0)(0.38,100.0)(0.4,100.0)(0.42,100.0)(0.44,100.0)(0.46,100.0)(0.48,100.0)(0.5,100.0)
    };
    \addlegendentry{$k=n-1$}
\end{axis}
\end{tikzpicture}
\end{subfigure}\begin{subfigure}{.5\textwidth} 
	\centering
\begin{tikzpicture}[scale=0.85]
\begin{axis}[
    title={(d) $n = 60$},
    xlabel={},
    ylabel={},
    xmin=0, xmax=0.5,
    ymin=0, ymax=100,
    xtick= {0,0.1,0.2,0.3,0.4,0.5},
    ytick={0,20,40,60,80,100},
    legend pos=south east,
    ymajorgrids=true,
	  xmajorgrids=true,
    grid style=dashed,
]

\addplot[
    color=black,
    mark size=3,
    mark=triangle*,
  	fill opacity=0.8,
  	draw opacity=0.8,
    ]
    coordinates {
    (0.0,0.0)(0.02,0.0)(0.04,0.0)(0.06,0.0)(0.08,0.0)(0.1,0.0)(0.12,0.0)(0.14,0.02)(0.16,0.13999999999999999)(0.18,1.3599999999999999)(0.2,6.34)(0.22,18.35)(0.24,36.78)(0.26,56.120000000000005)(0.28,73.82)(0.3,86.07000000000001)(0.32,92.92)(0.34,96.77)(0.36,98.68)(0.38,99.2)(0.4,99.78)(0.42,99.91)(0.44,99.96000000000001)(0.46,99.97)(0.48,100.0)(0.5,100.0)
    };
    \addlegendentry{$k=2$}

\addplot[
	    color=red,
        mark=square*,
        fill opacity=0.8,
        draw opacity=0.8,
	    ]
	    coordinates {
		(0.0,0.0)(0.02,1.5599999999999998)(0.04,10.03)(0.06,41.55)(0.08,76.58)(0.1,93.77)(0.12,98.57000000000001)(0.14,99.72)(0.16,99.94)(0.18,99.99)(0.2,99.99)(0.22,100.0)(0.24,100.0)(0.26,100.0)(0.28,100.0)(0.3,100.0)(0.32,100.0)(0.34,100.0)(0.36,100.0)(0.38,100.0)(0.4,100.0)(0.42,100.0)(0.44,100.0)(0.46,100.0)(0.48,100.0)(0.5,100.0)
		};
	    \addlegendentry{$k=3$}

\addplot[
    color=blue,
    mark=otimes*,
    fill opacity=0.8,
  	draw opacity=0.8,
    ]
    coordinates {
    (0.0,0.0)(0.02,16.46)(0.04,67.36)(0.06,92.02)(0.08,98.33)(0.1,99.41)(0.12,99.89)(0.14,99.96000000000001)(0.16,99.99)(0.18,100.0)(0.2,100.0)(0.22,100.0)(0.24,100.0)(0.26,100.0)(0.28,100.0)(0.3,100.0)(0.32,100.0)(0.34,100.0)(0.36,100.0)(0.38,100.0)(0.4,100.0)(0.42,100.0)(0.44,100.0)(0.46,100.0)(0.48,100.0)(0.5,100.0)
    };
    \addlegendentry{$k=4$}
\addplot[
    color=green,
    mark=diamond*,
		fill opacity=0.8,
  	draw opacity=0.8,
    ]
    coordinates {
    (0.0,0.0)(0.02,47.61)(0.04,81.67999999999999)(0.06,94.62)(0.08,98.45)(0.1,99.48)(0.12,99.9)(0.14,99.99)(0.16,99.99)(0.18,100.0)(0.2,100.0)(0.22,99.99)(0.24,100.0)(0.26,100.0)(0.28,100.0)(0.3,100.0)(0.32,100.0)(0.34,100.0)(0.36,100.0)(0.38,100.0)(0.4,100.0)(0.42,100.0)(0.44,100.0)(0.46,100.0)(0.48,100.0)(0.5,100.0)
    };
    \addlegendentry{$k=n-1$}
\end{axis}
\end{tikzpicture}
\end{subfigure}

\vspace{2mm}

	\begin{subfigure}{.5\textwidth} 
		\centering
\begin{tikzpicture}[scale=0.85]
\begin{axis}[
    title={(e) $n = 80$},
    xlabel={},
    ylabel={},
    xmin=0, xmax=0.5,
    ymin=0, ymax=100,
    xtick= {0,0.1,0.2,0.3,0.4,0.5},
    ytick={0,20,40,60,80,100},
    legend pos=south east,
    ymajorgrids=true,
		xmajorgrids=true,
    grid style=dashed,
]

\addplot[
    color=black,
    mark size=3,
    mark=triangle*,
  	fill opacity=0.8,
  	draw opacity=0.8,
    ]
    coordinates {
    (0.0,0.0)(0.02,0.0)(0.04,0.0)(0.06,0.0)(0.08,0.0)(0.1,0.0)(0.12,0.01)(0.14,0.09)(0.16,1.04)(0.18,6.72)(0.2,22.32)(0.22,46.69)(0.24,70.13000000000001)(0.26,85.3)(0.28,93.61)(0.3,97.78999999999999)(0.32,99.17)(0.34,99.69)(0.36,99.94)(0.38,99.96000000000001)(0.4,99.99)(0.42,99.99)(0.44,100.0)(0.46,100.0)(0.48,100.0)(0.5,100.0)
    };
    \addlegendentry{$k=2$}

\addplot[
        color=red,
        mark=square*,
        fill opacity=0.8,
        draw opacity=0.8,
	    ]
	    coordinates {
    (0.0,0.0)(0.02,2.19)(0.04,21.57)(0.06,69.78999999999999)(0.08,94.22)(0.1,99.41)(0.12,99.88)(0.14,100.0)(0.16,100.0)(0.18,100.0)(0.2,100.0)(0.22,100.0)(0.24,100.0)(0.26,100.0)(0.28,100.0)(0.3,100.0)(0.32,100.0)(0.34,100.0)(0.36,100.0)(0.38,100.0)(0.4,100.0)(0.42,100.0)(0.44,100.0)(0.46,100.0)(0.48,100.0)(0.5,100.0)	
		};
	    \addlegendentry{$k=3$}

\addplot[
    color=blue,
    mark=otimes*,
    fill opacity=0.8,
  	draw opacity=0.8,
    ]
    coordinates {
    (0.0,0.0)(0.02,30.37)(0.04,86.29)(0.06,98.21)(0.08,99.66000000000001)(0.1,99.95)(0.12,99.99)(0.14,99.99)(0.16,100.0)(0.18,100.0)(0.2,100.0)(0.22,100.0)(0.24,100.0)(0.26,100.0)(0.28,100.0)(0.3,100.0)(0.32,100.0)(0.34,100.0)(0.36,100.0)(0.38,100.0)(0.4,100.0)(0.42,100.0)(0.44,100.0)(0.46,100.0)(0.48,100.0)(0.5,100.0)	
    };
    \addlegendentry{$k=4$}
\addplot[
    color=green,
    mark=diamond*,
		fill opacity=0.8,
  	draw opacity=0.8,
    ]
    coordinates {
    (0.0,0.0)(0.02,63.39)(0.04,91.81)(0.06,98.34)(0.08,99.71)(0.1,99.97)(0.12,99.99)(0.14,99.99)(0.16,100.0)(0.18,100.0)(0.2,100.0)(0.22,100.0)(0.24,100.0)(0.26,100.0)(0.28,100.0)(0.3,100.0)(0.32,100.0)(0.34,100.0)(0.36,100.0)(0.38,100.0)(0.4,100.0)(0.42,100.0)(0.44,100.0)(0.46,100.0)(0.48,100.0)(0.5,100.0)
    };
    \addlegendentry{$k=n-1$}
\end{axis}
\end{tikzpicture}
\end{subfigure}\begin{subfigure}{.5\textwidth} 
	\centering
\begin{tikzpicture}[scale=0.85]
\begin{axis}[
    title={(f) $n = 100$},
    xlabel={},
    ylabel={},
    xmin=0, xmax=0.5,
    ymin=0, ymax=100,
    xtick= {0,0.1,0.2,0.3,0.4,0.5},
    ytick={0,20,40,60,80,100},
    legend pos=south east,
    ymajorgrids=true,
		xmajorgrids=true,
    grid style=dashed,
]

\addplot[
    color=black,
    mark size=3,
    mark=triangle*,
  	fill opacity=0.8,
  	draw opacity=0.8,
    ]
    coordinates {
    (0.0,0.0)(0.02,0.0)(0.04,0.0)(0.06,0.0)(0.08,0.0)(0.1,0.0)(0.12,0.01)(0.14,0.38999999999999996)(0.16,4.18)(0.18,20.57)(0.2,47.5)(0.22,73.08)(0.24,88.79)(0.26,95.81)(0.28,98.83999999999999)(0.3,99.67)(0.32,99.88)(0.34,99.99)(0.36,100.0)(0.38,100.0)(0.4,100.0)(0.42,100.0)(0.44,100.0)(0.46,100.0)(0.48,100.0)(0.5,100.0)
    };
    \addlegendentry{$k=2$}

\addplot[
        color=red,
        mark=square*,
        fill opacity=0.8,
        draw opacity=0.8,
	    ]
	    coordinates {
    (0.0,0.0)(0.02,2.63)(0.04,39.15)(0.06,88.94999999999999)(0.08,98.81)(0.1,99.92)(0.12,100.0)(0.14,100.0)(0.16,100.0)(0.18,100.0)(0.2,100.0)(0.22,100.0)(0.24,100.0)(0.26,100.0)(0.28,100.0)(0.3,100.0)(0.32,100.0)(0.34,100.0)(0.36,100.0)(0.38,100.0)(0.4,100.0)(0.42,100.0)(0.44,100.0)(0.46,100.0)(0.48,100.0)(0.5,100.0)
		};
	    \addlegendentry{$k=3$}

\addplot[
    color=blue,
    mark=otimes*,
    fill opacity=0.8,
  	draw opacity=0.8,
    ]
    coordinates {
    (0.0,0.0)(0.02,46.02)(0.04,94.67)(0.06,99.59)(0.08,99.97)(0.1,100.0)(0.12,100.0)(0.14,100.0)(0.16,100.0)(0.18,100.0)(0.2,100.0)(0.22,100.0)(0.24,100.0)(0.26,100.0)(0.28,100.0)(0.3,100.0)(0.32,100.0)(0.34,100.0)(0.36,100.0)(0.38,100.0)(0.4,100.0)(0.42,100.0)(0.44,100.0)(0.46,100.0)(0.48,100.0)(0.5,100.0)	
    };
    \addlegendentry{$k=4$}
\addplot[
    color=green,
    mark=diamond*,
		fill opacity=0.8,
  	draw opacity=0.8,
    ]
    coordinates {
    (0.0,0.0)(0.02,74.42999999999999)(0.04,96.41)(0.06,99.55000000000001)(0.08,99.97)(0.1,99.99)(0.12,100.0)(0.14,100.0)(0.16,100.0)(0.18,100.0)(0.2,100.0)(0.22,100.0)(0.24,100.0)(0.26,100.0)(0.28,100.0)(0.3,100.0)(0.32,100.0)(0.34,100.0)(0.36,100.0)(0.38,100.0)(0.4,100.0)(0.42,100.0)(0.44,100.0)(0.46,100.0)(0.48,100.0)(0.5,100.0)
    };
    \addlegendentry{$k=n-1$}
\end{axis}
\end{tikzpicture}
\end{subfigure}
\caption{Percentage of tournaments generated by the Condorcet random model with probability~$p$ in which all alternatives are $k$-kings, for $k\in\{2,3,4,n-1\}$ and different sizes $n$ of the tournament.
The horizontal and vertical axes correspond to the probability $p$ and the percentage, respectively. Percentages are taken over 10000 generated tournaments.}
\label{fig:condorcet-choose-all}
\end{figure}

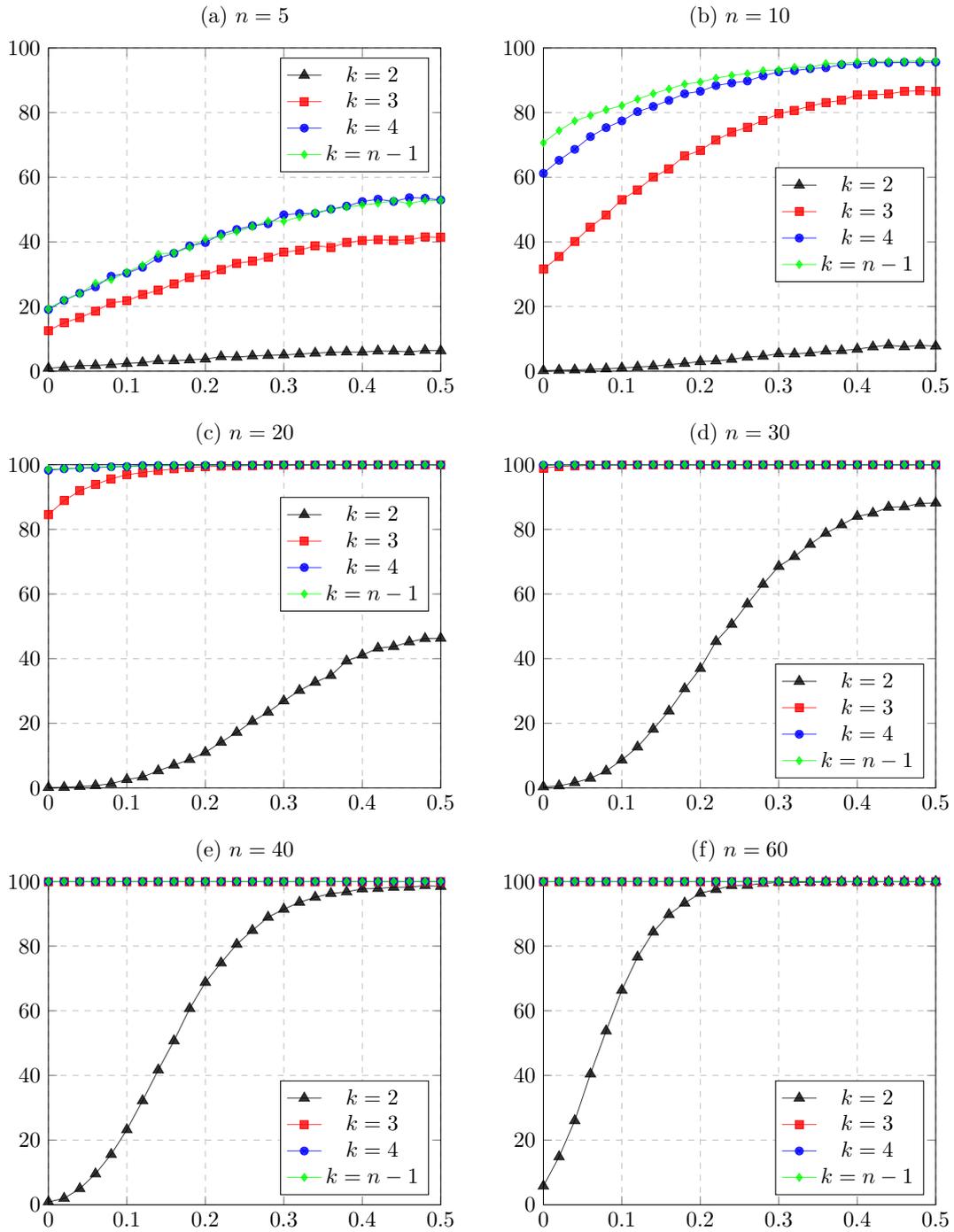
\begin{figure}
	\centering
	\begin{subfigure}{.5\textwidth} 
		\centering
\begin{tikzpicture}[scale=0.85]
\begin{axis}[
    title={(a) $n = 5$},
    xlabel={},
    ylabel={},
    xmin=0, xmax=0.5,
    ymin=0, ymax=100,
    xtick= {0,0.1,0.2,0.3,0.4,0.5},
    ytick={0,20,40,60,80,100},
    legend pos=north east,
    ymajorgrids=true,
		xmajorgrids=true,
    grid style=dashed,
]

\addplot[
    color=black,
	  mark size=3,
    mark=triangle*,
  	fill opacity=0.8,
  	draw opacity=0.8,
    ]
    coordinates {
    (0.0,0.89)(0.02,1.28)(0.04,1.69)(0.06,1.76)(0.08,2.02)(0.1,2.41)(0.12,2.62)(0.14,3.2399999999999998)(0.16,3.2199999999999998)(0.18,3.51)(0.2,3.74)(0.22,4.4799999999999995)(0.24,4.35)(0.26,4.72)(0.28,4.9)(0.3,5.01)(0.32,5.3100000000000005)(0.34,5.56)(0.36,5.84)(0.38,5.96)(0.4,5.87)(0.42,6.25)(0.44,6.1899999999999995)(0.46,5.96)(0.48,6.419999999999999)(0.5,6.3)
    };
    \addlegendentry{$k=2$}

\addplot[
	    color=red,
	    mark=square*,
			fill opacity=0.8,
  	  draw opacity=0.8,
	    ]
	    coordinates {
	    (0.0,12.53)(0.02,15.0)(0.04,16.520000000000003)(0.06,18.6)(0.08,21.07)(0.1,21.8)(0.12,23.68)(0.14,25.0)(0.16,27.0)(0.18,29.060000000000002)(0.2,29.82)(0.22,31.52)(0.24,33.39)(0.26,34.01)(0.28,35.3)(0.3,36.86)(0.32,37.38)(0.34,38.81)(0.36,38.32)(0.38,39.76)(0.4,40.42)(0.42,40.69)(0.44,40.510000000000005)(0.46,40.64)(0.48,41.57)(0.5,41.370000000000005)
	    };
	    \addlegendentry{$k=3$}

\addplot[
    color=blue,
    mark=otimes*,
		fill opacity=0.8,
  	draw opacity=0.8,
    ]
    coordinates {
    (0.0,19.07)(0.02,21.9)(0.04,24.12)(0.06,26.07)(0.08,29.39)(0.1,30.320000000000004)(0.12,32.14)(0.14,34.949999999999996)(0.16,36.480000000000004)(0.18,38.79)(0.2,39.79)(0.22,42.41)(0.24,43.85)(0.26,44.99)(0.28,45.61)(0.3,48.370000000000005)(0.32,48.79)(0.34,48.809999999999995)(0.36,50.129999999999995)(0.38,51.09)(0.4,52.470000000000006)(0.42,53.22)(0.44,52.51)(0.46,53.68000000000001)(0.48,53.510000000000005)(0.5,52.96999999999999)
    };
    \addlegendentry{$k=4$}
    
\addplot[
    color=green,
    mark=diamond*,
		fill opacity=0.8,
  	draw opacity=0.8,
    ]
    coordinates {
    (0.0,19.28)(0.02,22.05)(0.04,23.96)(0.06,27.21)(0.08,28.32)(0.1,30.680000000000003)(0.12,32.74)(0.14,36.21)(0.16,36.620000000000005)(0.18,38.3)(0.2,40.910000000000004)(0.22,41.77)(0.24,43.19)(0.26,44.73)(0.28,46.44)(0.3,46.400000000000006)(0.32,47.73)(0.34,49.059999999999995)(0.36,50.01)(0.38,50.79)(0.4,51.449999999999996)(0.42,51.959999999999994)(0.44,52.78)(0.46,51.89)(0.48,52.83)(0.5,52.849999999999994)
    };
    \addlegendentry{$k=n-1$}
\end{axis}
\end{tikzpicture}
\end{subfigure}\begin{subfigure}{.5\textwidth} 
	\centering
\begin{tikzpicture}[scale=0.85]
\begin{axis}[
    title={(b) $n = 10$},
    xlabel={},
    ylabel={},
    xmin=0, xmax=0.5,
    ymin=0, ymax=100,
    xtick= {0,0.1,0.2,0.3,0.4,0.5},
    ytick={0,20,40,60,80,100},
    legend style={at={(0.97,0.63)}},
    ymajorgrids=true,
		xmajorgrids=true,
    grid style=dashed,
]

\addplot[
    color=black,
    mark size=3,
    mark=triangle*,
  	fill opacity=0.8,
  	draw opacity=0.8,
    ]
    coordinates {
    (0.0,0.16999999999999998)(0.02,0.29)(0.04,0.37)(0.06,0.54)(0.08,0.7100000000000001)(0.1,0.97)(0.12,1.26)(0.14,1.5)(0.16,2.0)(0.18,2.42)(0.2,2.97)(0.22,3.17)(0.24,3.65)(0.26,4.36)(0.28,4.65)(0.3,5.42)(0.32,5.33)(0.34,5.62)(0.36,6.22)(0.38,6.35)(0.4,6.81)(0.42,7.5)(0.44,8.05)(0.46,7.539999999999999)(0.48,7.9799999999999995)(0.5,7.739999999999999)
    };
    \addlegendentry{$k=2$}

\addplot[
	    color=red,
        mark=square*,
        fill opacity=0.8,
        draw opacity=0.8,
	    ]
	    coordinates {
	    (0.0,31.540000000000003)(0.02,35.449999999999996)(0.04,40.12)(0.06,44.58)(0.08,48.36)(0.1,53.11)(0.12,56.11000000000001)(0.14,60.019999999999996)(0.16,62.63999999999999)(0.18,66.59)(0.2,68.32000000000001)(0.22,71.48)(0.24,73.89)(0.26,75.47)(0.28,77.57)(0.3,79.64)(0.32,80.7)(0.34,81.99)(0.36,83.1)(0.38,83.78999999999999)(0.4,85.41)(0.42,85.48)(0.44,85.75)(0.46,86.58)(0.48,86.74)(0.5,86.50999999999999)
    	   };
	    \addlegendentry{$k=3$}

\addplot[
    color=blue,
    mark=otimes*,
    fill opacity=0.8,
  	draw opacity=0.8,
    ]
    coordinates {
     (0.0,61.21)(0.02,65.25)(0.04,68.64)(0.06,72.57000000000001)(0.08,75.35)(0.1,77.42999999999999)(0.12,80.28)(0.14,81.91000000000001)(0.16,83.78)(0.18,85.82)(0.2,86.59)(0.22,88.36)(0.24,89.16)(0.26,89.77000000000001)(0.28,91.36999999999999)(0.3,92.58)(0.32,92.96)(0.34,93.58)(0.36,93.87)(0.38,94.81)(0.4,94.91000000000001)(0.42,95.45)(0.44,95.39)(0.46,95.56)(0.48,95.52000000000001)(0.5,95.61)
    };
    \addlegendentry{$k=4$}
    
\addplot[
    color=green,
    mark=diamond*,
		fill opacity=0.8,
  	draw opacity=0.8,
    ]
    coordinates {
    (0.0,70.64)(0.02,74.42999999999999)(0.04,77.41)(0.06,79.12)(0.08,80.88)(0.1,82.16)(0.12,84.14)(0.14,85.86)(0.16,87.3)(0.18,88.8)(0.2,89.44)(0.22,90.67)(0.24,91.49000000000001)(0.26,92.03)(0.28,92.97999999999999)(0.3,93.2)(0.32,94.02000000000001)(0.34,93.91000000000001)(0.36,95.17999999999999)(0.38,95.07)(0.4,95.77)(0.42,95.71)(0.44,95.85000000000001)(0.46,95.8)(0.48,96.02000000000001)(0.5,95.96000000000001)
    };
    \addlegendentry{$k=n-1$}
\end{axis}
\end{tikzpicture}
\end{subfigure}

\vspace{2mm}

	\centering
	\begin{subfigure}{.5\textwidth} 
		\centering
\begin{tikzpicture}[scale=0.85]
\begin{axis}[
    title={(c) $n = 20$},
    xlabel={},
    ylabel={},
    xmin=0, xmax=0.5,
    ymin=0, ymax=100,
    xtick= {0,0.1,0.2,0.3,0.4,0.5},
    ytick={0,20,40,60,80,100},
    legend style={at={(0.97,0.9)}},
    ymajorgrids=true,
		xmajorgrids=true,
    grid style=dashed,
]

\addplot[
    color=black,
	  mark size=3,
    mark=triangle*,
  	fill opacity=0.8,
  	draw opacity=0.8,
    ]
    coordinates {
    (0.0,0.13)(0.02,0.19)(0.04,0.48)(0.06,0.73)(0.08,1.3299999999999998)(0.1,2.6)(0.12,3.4000000000000004)(0.14,5.3)(0.16,7.08)(0.18,8.799999999999999)(0.2,11.0)(0.22,14.149999999999999)(0.24,17.19)(0.26,20.580000000000002)(0.28,23.45)(0.3,26.93)(0.32,30.15)(0.34,32.71)(0.36,34.78)(0.38,39.28)(0.4,41.11)(0.42,43.24)(0.44,43.730000000000004)(0.46,45.16)(0.48,46.21)(0.5,46.300000000000004)
    };
    \addlegendentry{$k=2$}

\addplot[
	    color=red,
        mark=square*,
        fill opacity=0.8,
        draw opacity=0.8,
	    ]
	    coordinates {
    (0.0,84.50999999999999)(0.02,88.88000000000001)(0.04,91.93)(0.06,93.88)(0.08,95.66)(0.1,96.89)(0.12,97.57000000000001)(0.14,98.19)(0.16,98.71)(0.18,99.21)(0.2,99.4)(0.22,99.56)(0.24,99.66000000000001)(0.26,99.77000000000001)(0.28,99.86)(0.3,99.88)(0.32,99.83)(0.34,99.92)(0.36,99.94)(0.38,99.96000000000001)(0.4,99.92999999999999)(0.42,99.98)(0.44,99.99)(0.46,99.98)(0.48,99.99)(0.5,99.97)
	    };
	    \addlegendentry{$k=3$}

\addplot[
    color=blue,
    mark=otimes*,
    fill opacity=0.8,
  	draw opacity=0.8,
    ]
    coordinates {
    (0.0,98.29)(0.02,98.72999999999999)(0.04,98.97)(0.06,99.06)(0.08,99.36)(0.1,99.44)(0.12,99.75)(0.14,99.76)(0.16,99.77000000000001)(0.18,99.8)(0.2,99.81)(0.22,99.87)(0.24,99.91)(0.26,99.92999999999999)(0.28,99.95)(0.3,99.99)(0.32,99.96000000000001)(0.34,99.97)(0.36,99.98)(0.38,100.0)(0.4,99.98)(0.42,100.0)(0.44,100.0)(0.46,99.99)(0.48,99.99)(0.5,99.98)
    };
    \addlegendentry{$k=4$}
\addplot[
    color=green,
    mark=diamond*,
		fill opacity=0.8,
  	draw opacity=0.8,
    ]
    coordinates {
    (0.0,98.61999999999999)(0.02,98.77)(0.04,99.00999999999999)(0.06,99.35000000000001)(0.08,99.31)(0.1,99.42999999999999)(0.12,99.61)(0.14,99.71)(0.16,99.77000000000001)(0.18,99.95)(0.2,99.85000000000001)(0.22,99.85000000000001)(0.24,99.92999999999999)(0.26,99.92999999999999)(0.28,99.91)(0.3,99.96000000000001)(0.32,99.96000000000001)(0.34,99.96000000000001)(0.36,99.98)(0.38,99.98)(0.4,100.0)(0.42,99.99)(0.44,99.98)(0.46,100.0)(0.48,99.99)(0.5,100.0)
    };
    \addlegendentry{$k=n-1$}
\end{axis}
\end{tikzpicture}
\end{subfigure}\begin{subfigure}{.5\textwidth} 
	\centering
\begin{tikzpicture}[scale=0.85]
\begin{axis}[
    title={(d) $n = 30$},
    xlabel={},
    ylabel={},
    xmin=0, xmax=0.5,
    ymin=0, ymax=100,
    xtick= {0,0.1,0.2,0.3,0.4,0.5},
    ytick={0,20,40,60,80,100},
    legend pos=south east,
    ymajorgrids=true,
	  xmajorgrids=true,
    grid style=dashed,
]

\addplot[
    color=black,
    mark size=3,
    mark=triangle*,
  	fill opacity=0.8,
  	draw opacity=0.8,
    ]
    coordinates {
    (0.0,0.33999999999999997)(0.02,0.65)(0.04,1.6500000000000001)(0.06,3.0)(0.08,5.2299999999999995)(0.1,8.64)(0.12,12.65)(0.14,18.16)(0.16,23.79)(0.18,30.7)(0.2,36.95)(0.22,45.34)(0.24,50.61)(0.26,56.92)(0.28,63.029999999999994)(0.3,68.56)(0.32,71.6)(0.34,75.39)(0.36,78.79)(0.38,81.43)(0.4,84.02)(0.42,85.00999999999999)(0.44,86.85000000000001)(0.46,86.96000000000001)(0.48,88.06)(0.5,88.17)
    };
    \addlegendentry{$k=2$}

\addplot[
	    color=red,
        mark=square*,
        fill opacity=0.8,
        draw opacity=0.8,
	    ]
	    coordinates {
		(0.0,98.97)(0.02,99.45)(0.04,99.62)(0.06,99.81)(0.08,99.91)(0.1,99.92999999999999)(0.12,99.95)(0.14,99.99)(0.16,99.98)(0.18,100.0)(0.2,100.0)(0.22,100.0)(0.24,100.0)(0.26,100.0)(0.28,100.0)(0.3,100.0)(0.32,100.0)(0.34,100.0)(0.36,100.0)(0.38,100.0)(0.4,100.0)(0.42,100.0)(0.44,100.0)(0.46,100.0)(0.48,100.0)(0.5,100.0)
		};
	    \addlegendentry{$k=3$}

\addplot[
    color=blue,
    mark=otimes*,
    fill opacity=0.8,
  	draw opacity=0.8,
    ]
    coordinates {
    (0.0,99.89)(0.02,99.94)(0.04,99.97)(0.06,99.94)(0.08,99.98)(0.1,100.0)(0.12,100.0)(0.14,100.0)(0.16,100.0)(0.18,100.0)(0.2,100.0)(0.22,100.0)(0.24,100.0)(0.26,100.0)(0.28,100.0)(0.3,100.0)(0.32,100.0)(0.34,100.0)(0.36,100.0)(0.38,100.0)(0.4,100.0)(0.42,100.0)(0.44,100.0)(0.46,100.0)(0.48,100.0)(0.5,100.0)
    };
    \addlegendentry{$k=4$}
\addplot[
    color=green,
    mark=diamond*,
		fill opacity=0.8,
  	draw opacity=0.8,
    ]
    coordinates {
    (0.0,99.92999999999999)(0.02,99.97)(0.04,99.97)(0.06,99.98)(0.08,100.0)(0.1,99.98)(0.12,99.96000000000001)(0.14,100.0)(0.16,100.0)(0.18,100.0)(0.2,99.99)(0.22,100.0)(0.24,100.0)(0.26,100.0)(0.28,100.0)(0.3,100.0)(0.32,100.0)(0.34,100.0)(0.36,100.0)(0.38,100.0)(0.4,100.0)(0.42,100.0)(0.44,100.0)(0.46,100.0)(0.48,100.0)(0.5,100.0)
    };
    \addlegendentry{$k=n-1$}
\end{axis}
\end{tikzpicture}
\end{subfigure}

\vspace{2mm}

	\begin{subfigure}{.5\textwidth} 
		\centering
\begin{tikzpicture}[scale=0.85]
\begin{axis}[
    title={(e) $n = 40$},
    xlabel={},
    ylabel={},
    xmin=0, xmax=0.5,
    ymin=0, ymax=100,
    xtick= {0,0.1,0.2,0.3,0.4,0.5},
    ytick={0,20,40,60,80,100},
    legend pos=south east,
    ymajorgrids=true,
		xmajorgrids=true,
    grid style=dashed,
]

\addplot[
    color=black,
    mark size=3,
    mark=triangle*,
  	fill opacity=0.8,
  	draw opacity=0.8,
    ]
    coordinates {
    (0.0,0.88)(0.02,1.97)(0.04,4.96)(0.06,9.54)(0.08,15.559999999999999)(0.1,23.21)(0.12,32.18)(0.14,41.69)(0.16,50.71)(0.18,60.67)(0.2,68.8)(0.22,74.81)(0.24,80.58)(0.26,84.87)(0.28,88.94999999999999)(0.3,91.43)(0.32,93.56)(0.34,95.12)(0.36,96.28)(0.38,96.78)(0.4,97.72999999999999)(0.42,97.84)(0.44,98.21)(0.46,98.25)(0.48,98.72)(0.5,98.45)
    };
    \addlegendentry{$k=2$}

\addplot[
        color=red,
        mark=square*,
        fill opacity=0.8,
        draw opacity=0.8,
	    ]
	    coordinates {
    (0.0,99.96000000000001)(0.02,100.0)(0.04,99.98)(0.06,99.99)(0.08,99.99)(0.1,99.99)(0.12,100.0)(0.14,100.0)(0.16,100.0)(0.18,100.0)(0.2,100.0)(0.22,100.0)(0.24,100.0)(0.26,100.0)(0.28,100.0)(0.3,100.0)(0.32,100.0)(0.34,100.0)(0.36,100.0)(0.38,100.0)(0.4,100.0)(0.42,100.0)(0.44,100.0)(0.46,100.0)(0.48,100.0)(0.5,100.0)	
		};
	    \addlegendentry{$k=3$}

\addplot[
    color=blue,
    mark=otimes*,
    fill opacity=0.8,
  	draw opacity=0.8,
    ]
    coordinates {
    (0.0,100.0)(0.02,100.0)(0.04,100.0)(0.06,100.0)(0.08,100.0)(0.1,100.0)(0.12,100.0)(0.14,100.0)(0.16,100.0)(0.18,100.0)(0.2,100.0)(0.22,100.0)(0.24,100.0)(0.26,100.0)(0.28,100.0)(0.3,100.0)(0.32,100.0)(0.34,100.0)(0.36,100.0)(0.38,100.0)(0.4,100.0)(0.42,100.0)(0.44,100.0)(0.46,100.0)(0.48,100.0)(0.5,100.0)	
    };
    \addlegendentry{$k=4$}
\addplot[
    color=green,
    mark=diamond*,
		fill opacity=0.8,
  	draw opacity=0.8,
    ]
    coordinates {
    (0.0,100.0)(0.02,100.0)(0.04,100.0)(0.06,100.0)(0.08,100.0)(0.1,100.0)(0.12,100.0)(0.14,100.0)(0.16,100.0)(0.18,100.0)(0.2,100.0)(0.22,100.0)(0.24,100.0)(0.26,100.0)(0.28,100.0)(0.3,100.0)(0.32,100.0)(0.34,100.0)(0.36,100.0)(0.38,100.0)(0.4,100.0)(0.42,100.0)(0.44,100.0)(0.46,100.0)(0.48,100.0)(0.5,100.0)
    };
    \addlegendentry{$k=n-1$}
\end{axis}
\end{tikzpicture}
\end{subfigure}\begin{subfigure}{.5\textwidth} 
	\centering
\begin{tikzpicture}[scale=0.85]
\begin{axis}[
    title={(f) $n = 60$},
    xlabel={},
    ylabel={},
    xmin=0, xmax=0.5,
    ymin=0, ymax=100,
    xtick= {0,0.1,0.2,0.3,0.4,0.5},
    ytick={0,20,40,60,80,100},
    legend pos=south east,
    ymajorgrids=true,
		xmajorgrids=true,
    grid style=dashed,
]

\addplot[
    color=black,
    mark size=3,
    mark=triangle*,
  	fill opacity=0.8,
  	draw opacity=0.8,
    ]
    coordinates {
    (0.0,5.7700000000000005)(0.02,14.81)(0.04,25.990000000000002)(0.06,40.44)(0.08,53.769999999999996)(0.1,66.36)(0.12,76.61)(0.14,84.38)(0.16,89.79)(0.18,93.31)(0.2,96.37)(0.22,97.53)(0.24,98.72999999999999)(0.26,98.82)(0.28,99.44)(0.3,99.74)(0.32,99.72)(0.34,99.78)(0.36,99.89)(0.38,99.95)(0.4,99.99)(0.42,99.92999999999999)(0.44,100.0)(0.46,99.98)(0.48,99.99)(0.5,100.0)
    };
    \addlegendentry{$k=2$}

\addplot[
        color=red,
        mark=square*,
        fill opacity=0.8,
        draw opacity=0.8,
	    ]
	    coordinates {
    (0.0,100.0)(0.02,100.0)(0.04,100.0)(0.06,100.0)(0.08,100.0)(0.1,100.0)(0.12,100.0)(0.14,100.0)(0.16,100.0)(0.18,100.0)(0.2,100.0)(0.22,100.0)(0.24,100.0)(0.26,100.0)(0.28,100.0)(0.3,100.0)(0.32,100.0)(0.34,100.0)(0.36,100.0)(0.38,100.0)(0.4,100.0)(0.42,100.0)(0.44,100.0)(0.46,100.0)(0.48,100.0)(0.5,100.0)
		};
	    \addlegendentry{$k=3$}

\addplot[
    color=blue,
    mark=otimes*,
    fill opacity=0.8,
  	draw opacity=0.8,
    ]
    coordinates {
    (0.0,100.0)(0.02,100.0)(0.04,100.0)(0.06,100.0)(0.08,100.0)(0.1,100.0)(0.12,100.0)(0.14,100.0)(0.16,100.0)(0.18,100.0)(0.2,100.0)(0.22,100.0)(0.24,100.0)(0.26,100.0)(0.28,100.0)(0.3,100.0)(0.32,100.0)(0.34,100.0)(0.36,100.0)(0.38,100.0)(0.4,100.0)(0.42,100.0)(0.44,100.0)(0.46,100.0)(0.48,100.0)(0.5,100.0)	
    };
    \addlegendentry{$k=4$}
\addplot[
    color=green,
    mark=diamond*,
		fill opacity=0.8,
  	draw opacity=0.8,
    ]
    coordinates {
    (0.0,100.0)(0.02,100.0)(0.04,100.0)(0.06,100.0)(0.08,100.0)(0.1,100.0)(0.12,100.0)(0.14,100.0)(0.16,100.0)(0.18,100.0)(0.2,100.0)(0.22,100.0)(0.24,100.0)(0.26,100.0)(0.28,100.0)(0.3,100.0)(0.32,100.0)(0.34,100.0)(0.36,100.0)(0.38,100.0)(0.4,100.0)(0.42,100.0)(0.44,100.0)(0.46,100.0)(0.48,100.0)(0.5,100.0)	
    };
    \addlegendentry{$k=n-1$}
\end{axis}
\end{tikzpicture}
\end{subfigure}
\caption{Percentage of tournaments generated by the gap model with probability~$p$ in which all alternatives are $k$-kings, for $k\in\{2,3,4,n-1\}$ and different sizes $n$ of the tournament.
The horizontal and vertical axes correspond to the probability $p$ and the percentage, respectively. Percentages are taken over 10000 generated tournaments.}
\label{fig:gap-choose-all}
\end{figure}

\end{document}